\documentclass[11pt]{article}

\usepackage{amsmath,amssymb,amsthm,amscd,eucal}
\usepackage{hyperref}
\usepackage{latexsym}
\usepackage{graphicx}
\usepackage[usenames]{color} 
\usepackage[dvipsnames]{xcolor}
\usepackage{times,tikz}
\usetikzlibrary{backgrounds}
\usepackage{paralist}
\usepackage{genyoungtabtikz}
\usepackage{arydshln}

\newtheorem{thm}[equation]{Theorem}
\newtheorem{cor}[equation]{Corollary}

\newtheorem{prop}[equation]{Proposition}
\theoremstyle{definition}
  
\newtheorem{remark}[equation]{Remark}
\newtheorem{example}[equation]{Example}
 
\numberwithin{equation}{section}

\newcommand{\dimm}{\mathsf{dim}\,} 
\newcommand{\para}{\xi}
\newcommand{\qds}{\mathsf{qd}}
\newcommand{\half}{\frac{1}{2}}
\newcommand{\Ind}{\mathsf{Ind}} 
\newcommand{\Res}{\mathsf{Res}} 
\newcommand{\ot}{\otimes} 
\newcommand{\End}{\mathsf{End}} 
\newcommand{\GG}{\mathsf{G}} 
\newcommand{\HH}{\mathsf{H}}  
\newcommand{\KK}{\mathsf{K}}

\newcommand{\perm}{\mathcal{M}}  
\newcommand{\UU}{\mathsf{U}}  
\newcommand{\VV}{\mathsf{V}}  
\newcommand{\WW}{\mathsf{W}}  
\newcommand{\xx}{\mathsf{X}}  
\newcommand{\yy}{\mathsf{Y}}
\newcommand{\Z}{\mathsf{Z}} 
\newcommand{\WZ}{\widehat{\Z}} 
\newcommand{\QP}{\mathsf{QP}} 
\newcommand{\QZ}{\mathsf{QZ}} 
\newcommand{\WQZ}{\widehat{\mathsf {QZ}}} 
\renewcommand{\S}{\mathsf{S}} 
 
\newcommand{\Rs}{\mathsf{R}} 
\newcommand{\CC}{\mathbb C}
\newcommand{\ZZ}{\mathbb Z}
\newcommand{\A}{\mathsf{A}} 
\newcommand{\modu}{\mathsf{M}_n}
\newcommand{\lam}{\lambda}
\newcommand{\al}{\alpha} 
\newcommand{\lb}{\left \{  }
\newcommand{\rb}{ \right \} }
\renewcommand{\P}{\mathsf{P}}
\newcommand{\ds}{\mathsf{d}}
 
\newcommand{\Ms}{\mathfrak M}

\newcommand{\RSK}{ \leftarrow }

\newcommand{\bsym}{\boldsymbol}

  \font \nmm = cmbx12

\usepackage[letterpaper,margin=1.25in]{geometry}

\title{\textmd{Dimensions of  irreducible  modules 
for partition algebras and} \\
\textmd {tensor power multiplicities for  \\
symmetric and alternating groups}}
\author{\textsc{Georgia Benkart} \\
\textit{\small Department of Mathematics} \\ 
\textit{\small University of Wisconsin-Madison}  \\
\textit{\small Madison, WI 53706, USA}\\
\texttt{\small benkart@math.wisc.edu}
\and
\textsc{Tom Halverson\footnote{
The second author gratefully acknowledges partial support from Simons Foundation grant 283311.}}\\
\textit{\small Department of Mathematics, Statistics, and Computer Science}\\
\textit{\small Macalester College}\\
\textit{\small Saint Paul, MN 55105, USA}\\
\texttt{\small halverson@macalester.edu}
\and
\textsc{Nate Harman}\footnote{The third author was partially supported by National Science
Foundation Graduate Research Fellowship grant 1122374.} \\
 \textit{\small Department of Mathematics} \\ 
\textit{\small Massachusetts Institute of Technology}  \\
\textit{\small Cambridge, MA , 02139, USA}\\
\texttt{\small nharman@math.mit.edu}} 
\date{May 27, 2016} 

\begin{document} 
\maketitle   
\begin{abstract}   
The partition algebra $\P_k(n)$ and the symmetric group $\S_n$ are in Schur-Weyl duality on the $k$-fold tensor power $\modu^{\otimes k}$ of the permutation module $\modu$ of $\S_n$, so there is a surjection $\P_k(n) \to \Z_k(n) := \End_{\S_n}(\modu^{\ot k}),$ which is an isomorphism when $n \ge 2k$. 
We prove a dimension formula for the irreducible modules of the centralizer algebra $\Z_k(n)$ in terms of Stirling numbers of the second kind. Via Schur-Weyl duality, these dimensions equal the multiplicities of the irreducible $\S_n$-modules in $\modu^{\otimes k}$. Our dimension expressions hold for any 
$n \geq 1$  and $k\ge0$. Our methods are based on an analog of Frobenius reciprocity that we show holds for the centralizer algebras of arbitrary  finite groups and their subgroups acting on a finite-dimensional module.  This enables us to 
generalize the above result to various analogs of the partition algebra including the centralizer algebra for the alternating group
acting on $\modu^{\ot k}$  and the quasi-partition algebra corresponding to tensor powers of  the reflection representation of $\S_n$.
\end{abstract}



\section{Introduction} 
The partition algebras $\P_k(\para)$, $\para \in \CC$,  were introduced by Martin (\cite{M1},\cite{M2},\cite{M3})   
 to study the  Potts lattice model of interacting spins in statistical mechanics.  
As shown by  Jones \cite{J},  there is a Schur-Weyl duality between the 
partition algebra $\P_k(n)$ and the symmetric group $\S_n$ acting as centralizers
of each other on the $k$-fold tensor power $\modu^{\ot k}$  of the $n$-dimensional permutation module
$\modu$ for $\S_n$ over $\CC$.    The surjective algebra homomorphism given in \cite{J} (see also \cite[Thm.\,3.6]{HR}), 
$$\P_k(n) \rightarrow  \Z_k(n) := \End_{\S_n}(\modu^{\ot k}) = \{T \in \End(\modu^{\ot k}) \mid
T (\sigma.u) = \sigma.T(u) \ \  \forall  \sigma \in \S_n, u \in \modu^{\ot k}\},$$  
is an isomorphism when $n \geq 2k$.     
 
The partition algebra $\P_k(\para)$  for $k \geq 1$ has a basis over $\CC$ indexed by set partitions
of the set $\{1,2,\ldots,2k\}$ into disjoint nonempty blocks.  An example of such a set partition
for $k = 7$ is $\big\{1,9,11 \mid  2,13 \mid 3 \mid 4,7,8\mid 5,6,12,14\big\}$, which has 5 blocks. 
The Stirling number $\lb 2k \atop r\rb$ counts the number of ways to partition $2k$ objects into $r$ nonempty 
disjoint blocks, so it follows that  
\vspace{-.25cm}
$$\dimm \P_k(\para) = \sum_{r=1}^{2k}  \lb 2k \atop r\rb = \mathsf{B}(2k), \quad \hbox {\rm(the \ $(2k)$th Bell number).} 
$$
 
  In $\P_{k+1}(\para)$,  the basis elements indexed by set partitions which have $k+1$ and $2(k+1)$ in the same block 
 form a subalgebra $\P_{k+\half}(\para)$ with $\dimm \P_{k+\half}(\para) = \mathsf{B}(2k+1)$.   If we regard  $\modu$ as a module for the symmetric group
$\S_{n-1}$ by restriction,   
there is a surjective algebra homomorphism $\P_{k+\half}(n) \rightarrow \Z_{k+\half}(n) :=  \End_{\S_{n-1}}(\modu^{\ot k})$, 
which is an isomorphism if $n \geq 2k$.   These intermediate algebras play an important role in 
understanding the structure and representation theory  of partition algebras (see for example, \cite{MR,HR}),
and they are a crucial component of the work in this paper.

The irreducible modules  for $\P_k(n)$ and $\P_{k+\half}(n)$ are labeled by partitions 
$\nu$  of $r$, where $r$ is an integer satisfying  $0 \leq r \leq k$.   Since the irreducible modules  $\S_n^\lambda$
for $\S_n$ are indexed by partitions $\lambda$ of $n$,  Schur-Weyl duality implies that  
 the  irreducible modules for  $\Z_k(n)$ are also indexed by partitions $\lambda$ of  $n$,  and for
$\Z_{k+\half}(n)$, by partitions $\mu$ of $n-1$.   The modules $\S_n^\lam$ (resp. $\S_{n-1}^\mu$)
occurring in $\modu^{\ot k}$ are indexed by partitions with the property that
the partition $\nu = \lambda^\#$  (resp. $\nu = \mu^\#$)  that results from  
deleting the largest part  of $\lambda$  (resp. of $\mu$)  must satisfy  $0 \leq \vert \nu \vert \leq k$,
where $\vert \nu \vert$ is the sum of the parts of $\nu$.    

In this paper,  we
\begin{itemize} 
\item establish general restriction/induction results for centralizer algebras,  proving in Theorem \ref{T:res-ind} that
an analog of  Frobenius
reciprocity for groups holds for  their centralizer algebras; 
\item give  restriction/induction Bratteli diagrams for the symmetric group-subgroup pair 
$(\S_n, \S_{n-1})$ and for the alternating group-subgroup pair  $(\A_n, \A_{n-1})$;
\item use  the reciprocity results to determine expressions for the dimensions of the irreducible modules  for 
$\Z_k(n)$, and $\Z_{k+\half}(n)$ (in Theorem \ref{T:cent}(a) and (b)),  and for $\P_k(\para)$, $\P_{k+\half}(\para)$ (in
Corollary \ref{C:partdims});  
 \item  determine the dimensions of the centralizer algebras $\Z_k(n)$ and $\Z_{k+\half}(n)$ (in Theorem~\ref{T:cent}(c) and (d));
 \item  give a combinatorial proof in Section \ref{ss:bijection}  of the dimension formula in Theorem \ref{T:cent}(a) by exhibiting a bijection between paths in the Bratteli diagram (aka vacillating tableaux) and pairs $(P,T)$ consisting of set partitions $P$ of $\{1,\dots,k\}$ and semistandard tableaux $T$ whose entries depend on $P$ (this bijection holds for all $k \geq 0$ and $n \geq 1$ and extends the one in \cite{CDDSY}, which was valid for $n \geq 2k$);
 \item  apply the restriction/induction results to the pair $(\S_n,\A_n)$  (resp. $(\S_{n-1},\A_{n-1})$)  
to determine the dimensions of the irreducible 
 modules for the centralizer algebras $\WZ_k(n):=\End_{\A_n}(\modu^{\ot k})$ and $\WZ_{k+\half}(n): =
 \End_{\A_{n-1}}(\modu^{\ot k})$  (in Theorem \ref{T:Ardims} (a) and (b));
 \item  determine dimension formulas for the centralizer algebras $\WZ_{k}(n)$ and $\WZ_{k+\half}(n)$  
 (Theorem \ref{T:Ardims}); 
 \item compute the dimensions of the irreducible modules for the centralizer algebras
 $\QZ_k(n): = \End_{\S_n}(\Rs_n^{\ot k})$, where $\Rs_n : =  \S_n^{[n-1,1]}$ is the $(n-1)$-dimensional
 irreducible reflection module of $\S_n$ corresponding to the partition $[n-1,1]$ of $n$,  
and for  their relatives $\QZ_{k+\half}(n)$, 
 $\WQZ_k(n)$,  and $\WQZ_{k+\half}(n)$ (in Theorem~\ref{T:QPcent}) and give Bratteli diagrams corresponding to $\Rs_n$ for  the group-subgroup pairs $(\S_n, \S_{n-1})$ and  $(\A_n, \A_{n-1})$. 
\end{itemize} 
By Schur-Weyl duality, the dimension of an irreducible module for the centralizer algebra equals the 
multiplicity of the corresponding irreducible module for the group. Consequently, our dimension formulas also
\begin{itemize}
\item determine the multiplicities of irreducible modules for $\S_n, \S_{n-1}, \A_n,$ and  $\A_{n-1}$ in $\modu^{\ot k}$ and in $\mathsf{R}_n^{\ot k}$ for
all $n,k \in \ZZ_{>0}$ (in Theorems \ref{T:cent} (a),(b), \ref{T:Ardims}, and \ref{T:QPcent}).
\end{itemize}

A preliminary version of this paper \cite{BH1}, posted on the arXiv by the first two authors, 
established dimension formulas for centralizer algebras as alternating sums  of expressions  involving Stirling numbers of the second kind and the number of standard tableaux compatible with certain \emph{$r$-sequences} of partitions for $\lambda$.  Upon seeing this result, the third author suggested the approach that we adopt in this paper for Theorem \ref{T:cent}\,(a).    As a consequence of this 
alternate way of computing the dimensions of the irreducible modules for $\Z_k(n)$,  we are able in this work to express all 
of the dimension formulas as positive sums using Stirling numbers of the second kind and Kostka numbers. It remains an open question to determine the relation between these two approaches.

The dimensions of the centralizer algebras $\Z_k(n)$ and $\WZ_k(n)$  were determined previously and can be found in   \cite{HR} and \cite{Bl1,Bl2},
respectively.
In this work, they are direct consequences of  the dimension formulas for the irreducible modules.  This
is a general phenomenon:   If $\Z_{k}(\GG) := \End_{\GG}(\xx^{\ot k})$  for a self-dual  
module $\xx$ of a group $\GG$, then  $\dimm \Z_k(\GG) = \mathsf{dim}\left(\xx^{\ot 2k}\right)^\GG$, where
$\left(\xx^{\ot 2k}\right)^\GG$ is the space of  $\GG$-invariants in $\xx^{\ot 2k}$.  Therefore, $\dimm \Z_{k}(\GG)$ is  the multiplicity of the trivial $\GG$-module $\GG^{\bullet}$
 in $\xx^{\ot 2k}$; equivalently, by Schur-Weyl duality,  it  is the dimension of the irreducible module associated 
 to $\GG^{\bullet}$  for  the centralizer algebra $\Z_{2k}(\GG)$ (see Section 2 for details).

Motivated by the work of  Goupil and Chauve \cite{GC} on Kronecker tableaux and Kronecker coefficients, Daugherty and Orellana in \cite{DO}
introduced the  {\it quasi-partition algebras} $\mathsf{QP}_k(\para)$, $\para \in \CC$,  and showed
that  there is  a surjection $\mathsf{QP}_k(n) \rightarrow \QZ_k(n)
= \End_{\S_n}(\Rs_n^{\ot k})$ for  $\Rs_n = \S_n^{[n-1,1]}$,  which 
is an isomorphism when $n \ge 2k$.   The dimensions for the irreducible modules for $\mathsf{QP}_k(\para)$,  with
$\para$ generic, are the same as the dimensions for $n \geq 2k$, and so are given by  the dimension formulas in Section 7 below.   These expressions differ from the ones that appear 
in \cite{DO}, which were based on results in \cite{GC}  and  hold whenever $n \geq 2k$,  as  the ones in Section 7 are  valid for all $k$ and $n$.

Using exponential generating functions from \cite{GC},  Ding  \cite{D} derived a formula for  the multiplicity of the irreducible $\S_n$-module $\S_n^\lam$ indexed by the partition
$\lam=[\lam_1,\dots,\lam_n]$  in $\modu^{\ot k}$ when $1 \leq k \leq n-\lambda_2$ 
and used that to obtain an expression for the multiplicity of  $\S_n^\lam$ in tensor powers 
$\Rs_n^{\ot k}$ of its reflection module $\Rs_n = \S_n^{[n-1,1]}$. 
The first  is a special case of part (a) of Theorem \ref{T:cent}  below and the
second a special case of Theorem \ref{T:QPcent}. As shown in \cite[Sec.~3]{D}, when $1 \le k \leq n-\lam_2$, these multiplicity formulas can be used 
to bound the mixing time of a Markov chain on $\S_n$.

\medskip 
\section{Restriction/Induction and Dimensions}\label{sec:restrictioninduction}
We begin with some general results on restriction and induction for centralizer algebras 
and then apply these results to the group-subgroup pairs   $(\S_n,\S_{n-1})$, $(\S_n,\A_n)$, and $(\A_n,\A_{n-1})$
acting on the $k$-fold tensor power of the $n$-dimensional permutation module $\modu$.  This will enable us to determine
the dimension of the centralizer algebras and their irreducible modules. 
 

Suppose $\GG$ is a finite group and $\HH$ is a subgroup of $\GG$.    Assume $\{\GG^\lam\}_{\lam \in \Lambda_\GG}$
and $\{\HH^\alpha\}_{\alpha\in \Lambda_\HH}$ are  the corresponding sets of irreducible modules for
these groups over $\mathbb C$.    We suppose that the restriction from $\GG$ to $\HH$ on $\GG^\lam$ is given by
\begin{equation}\label{eq:res} \mathsf{Res}_\HH^\GG(\GG^\lam) = \bigoplus_{\alpha \in \Lambda_\HH} c_\alpha^\lam \  \HH^\al. \end{equation} 
Then by Frobenius reciprocity, induction from $\HH$ to $\GG$ is given by 
\begin{equation}\label{eq:ind} \mathsf{Ind}_{\HH}^{\GG}(\HH^\alpha) =  \bigoplus_{\lam \in \Lambda_\GG } c_\alpha^\lam \ \GG^\lam.\end{equation} 

Assume now that  $\xx$ is a finite-dimensional $\GG$-module,  and consider the centralizer algebra
 $\Z_\xx(\GG) = \End_{\GG}(\xx) = \{T \in \End(\xx) \mid T(g.x) = g.T(x), \ \forall g \in \GG, x \in \xx\}$.  Regarding  $\xx$ as a module for the subgroup $\HH$ of $\GG$ by restriction,
we have reverse inclusion of the centralizer algebras
$\Z_\xx(\GG) \subseteq  \Z_\xx(\HH) = \End_{\HH}(\xx)$.    Let $\Lambda_{\xx,\GG}$ (resp. 
$\Lambda_{\xx,\HH}$) denote the subset of $\Lambda_\GG$  (resp. of $\Lambda_\HH$) 
 corresponding to the  irreducible $\GG$-modules (resp. $\HH$-modules)  which occur in $\xx$ with
multiplicity at least one.  
Then  \emph{Schur-Weyl duality} implies the following: 
  \begin{itemize}
  \item  the irreducible $\Z_\xx(\GG)$-modules $\Z_{\xx,\GG}^\lam$  are in bijection with the elements of $\lam \in \Lambda_{\xx,\GG}$;
 \item   the decomposition of \ $\xx$ into irreducible $\GG$-modules  is given by  
\begin{equation}\label{eq:X2G} \xx  \,\cong \,  \bigoplus_{\lam \in \Lambda_{\xx,\GG}}   \mathsf{d}_{\xx,\GG}^\lam \ \GG^\lam, \hskip.35in  
\hbox{where \ $\mathsf{d}_{\xx,\GG}^\lam = \dimm \Z_{\xx,\GG}^\lam$};
\end{equation} 
\item  the decomposition of \ $\xx$ into irreducible $\Z_\xx(\GG)$-modules
is given by   
\begin{equation}\label{eq:X2Z} \xx \, \cong \,  \bigoplus_{\lam \in \Lambda_{\xx,\GG}}
  \mathsf {d}_{\GG^\lam} \,  \Z_{\xx,\GG}^\lam, \hskip.3in
\hbox{where \ $\mathsf{d}_{\GG^\lam} = \dimm \GG^\lam;$} \hskip.15in
  \end{equation} 

\item  as a bimodule for \ $\GG \times \Z_\xx(\GG)$, 
\begin{equation}\label{eq:bimod}   \xx   \cong  \bigoplus_{\lam \in \Lambda_{\xx,\GG}}   \big( \GG^\lam \ot 
\Z_{\xx,\GG}^\lam \big);
 \end{equation} 
  \item $\Z_\xx(\GG)$ is a finite-dimensional semisimple associative algebra  and 
\begin{equation}\label{eq:sumofsquares}
 \dimm\Z_\xx(\GG) = \sum_{\lam \in \Lambda_{\xx,\GG}}  \big(\dimm \, \Z_{\xx,\GG}^\lam\big)^2 =  \sum_{\lam \in \Lambda_{\xx,\GG}}  (\mathsf{d}_{\xx,\GG}^\lam)^2.
\end{equation}
\end{itemize}

There is a corresponding Frobenius reciprocity  for centralizer algebras of the group-subgroup pair $(\GG,\HH)$, as indicated in the next result.
\medskip 

\begin{thm}\label{T:res-ind}  For  a finite-dimensional  $\GG$-module $\xx$,  let $\Z_\xx(\GG) = \End_{\GG}(\xx)$
and $\Z_\xx(\HH) = \End_{\HH}(\xx)$.
Let $\Lambda_{\xx,\GG}$ (resp. $\Lambda_{\xx,\HH}$)  be the set of indices $\lam \in \Lambda_\GG$ (resp. $\alpha  \in \Lambda_\HH$)  such that $\GG^\lam$ (resp. $\HH^\alpha$)  occurs in $\xx$ with multiplicity $\geq 1$,
and let $\Z_{\xx,\GG}^{\lam}$ (resp. $\Z_{\xx,\HH}^{\al}$)
denote the corresponding irreducible $\Z_{\xx}(\GG)$-module  (resp. $\Z_{\xx}(\HH)$-module).    
Assume  $c_\alpha^\lam$ is as in  \eqref{eq:res} and \eqref{eq:ind} above.  
 Then the following hold: 
 \begin{itemize} \item[{\rm (a)}]    $\Res_{\Z_\xx(\GG)}^{\Z_\xx(\HH)} \big(\Z_{\xx,\HH}^\al \big) =  \displaystyle{\bigoplus_{\lam \in \Lambda_{\xx,\GG}}} \,  c_\alpha^\lam \ \Z_{\xx,\GG}^\lam$.
  \item[{\rm (b)}]     For $\alpha \in \Lambda_{\xx, \HH}$, \  $\ds_{\xx,\HH}^\al  = \displaystyle{\sum_{\lam \in \Lambda_{\xx,\GG}}}\  c_\alpha^\lam \ 
 \ds_{\xx,\GG}^\lam$,\quad where \ $\ds_{\xx,\HH}^\al: = \dimm \Z_{\xx,\HH}^\al$ \  and \ $\ds_{\xx,\GG}^\lam :=   \dimm \Z_{\xx,\GG}^\lam$. 
\item[{\rm (c)}]   $\Ind_{\Z_{\xx}(\GG)}^{\Z_{\xx}(\HH)}\big(\Z_{\xx,\GG}^\lam \big) \, := \ \Z_{\xx}(\HH) \ot_{\Z_{\xx}(\GG)} 
\Z_{\xx,\GG}^\lam \ = \  \displaystyle{\bigoplus_{\al \in \Lambda_{\xx,\HH}}} \, c_\alpha^\lam \ \Z_{\xx,\HH}^\al.$  
\item[{\rm (d)}]  As a $\Z_{\xx}(\GG)$-module (via multiplication 
on the left),   $$\Z_{\xx}(\HH) = \bigoplus_{\lam \in \Lambda_{\xx,\GG}} \Bigg(\sum_{ \al \in \Lambda_{\xx,\HH}}\, 
c_\alpha^\lam \  {\ds}_{\xx,\HH}^\al \Bigg)\Z_{\xx,\GG}^\lam.$$ 
\item[{\rm (e)}] Assume $\yy$ is an  $\HH$-module and set  $\xx = \mathsf{Ind}_{\HH}^{\GG}(\yy)$. 
Let $ \Z_{\yy}(\HH) = \End_{\HH}(\yy)$,  and let  $\Z_{\yy,\HH}^\al$, $\al \in \Lambda_{\yy,\HH}$,  be the
irreducible $\Z_{\yy}(\HH)$-modules.   Then  
for \, $\lam \in \Lambda_{\xx, \GG}$,  \ 
$$\dimm \Z_{\xx,\GG}^\lam = \sum_{\al \in \Lambda_{\yy,\HH}} \ c_\al^\lam \ \dimm \Z_{\yy,\HH}^\al.$$
 \end{itemize}    \end{thm}

\begin{proof} \ (a) and (b): \   By Schur-Weyl duality,    $\xx \cong \bigoplus_{\lam \in \Lambda_{\xx,\GG}}  \left(\GG^\lam \ot \Z_{\xx,\GG}^\lam \right)$
as a $(\GG \times \Z_{\xx}(\GG))$-bimodule.   Therefore,  as an  $(\HH  \times \Z_{\xx}(\GG))$-bimodule,
\begin{equation}\label{eq:bimod2} \xx  \cong \sum_{\lam \in \Lambda_{\xx,\GG}} \bigg( \sum_{\al \in \Lambda_\HH}  c_\al^\lam \  \HH^\al \bigg) \ot  \Z_{\xx,\GG}^\lam
 \cong \sum_{\al \in \Lambda_\HH} \HH^\al \ot  \bigg(\sum_{\lam \in \Lambda_{\xx,\GG}} c_\al^\lam \   \Z_{\xx,\GG}^{\lam} \bigg). \end{equation}  
This says that the $\HH$-module \, $\HH^\al$ \, occurs as a summand of \, $\xx$\, with multiplicity equal to  \break
$\sum_{\lam \in \Lambda_{\xx,\GG}} \  c_\al^\lam \  \dimm \Z_{\xx,\GG}^\lam$.  
But from the decomposition 
\begin{equation}\label{eq:hdecomp} \xx  \cong  \bigoplus_{\al \in \Lambda_{\xx,\HH}} \bigg ( \HH^\al  \ot \Z_{\xx,\HH}^{\al} \bigg),\end{equation}
 we know that the only $\HH$-summands occurring in $\xx$ are those 
with $\al \in \Lambda_{\xx,\HH}$,  and $\HH^\al$ has multiplicity  $\dimm \Z_{\xx,\HH}^\al$ in $\xx$. 
Therefore, the sum in \eqref{eq:bimod2}  must be over $\al \in \Lambda_{\xx,\HH}$, and we have 
$\dimm \Z_{\xx,\HH}^\al = \sum_{\lam \in \Lambda_{\xx,\GG}} c_\al^\lam \  \dimm \Z_{\xx,\GG}^\lam$,
as claimed in (b).    Moreover,  the restriction of the  $\big(\HH \times \Z_{\xx}(\HH)\big)$-decomposition of $\xx$ in \eqref{eq:hdecomp}
to  $\HH \times  \Z_{\xx}(\GG)$ 
gives $\xx \cong  \bigoplus_{\al \in \Lambda_{\xx,\HH}}  \HH^\al  \ot \mathsf{Res}_{\Z_\xx(\GG)}^{\Z_\xx(\HH)} \left(\Z_{\xx,\HH}^{\al}\right).$   Since the decomposition of $\xx$
as a  $\big(\HH \times  \Z_{\xx}(\GG)\big)$-bimodule is unique,
 $\Res_{\Z_\xx(\GG)}^{\Z_\xx(\HH)}\left(\Z_{\xx,\HH}^\al \right)$  = $ \bigoplus_{\lam \in \Lambda_{\xx,\GG}}  c_\al^\lam \   \Z_{\xx,\GG}^\lam$ must hold,  as asserted in (a).  
Note that part (b) is just the dimension version of this relation.

For part (c),   we use the following standard result.  Assume $\mathrm{A}$ is an algebra and $\mathrm{B}$ is a subalgebra of 
$\mathrm{A}$.  
Let $\mathrm{W}$ be an $\mathrm{A}$-module and $\mathrm{V}$ be a $\mathrm{B}$-module.    Then 
\begin{equation}\label{eq:stand}\mathsf{Hom}_{\mathrm A}(\mathrm{A} \ot_{\mathrm B} \mathrm V, \mathrm W) = 
\mathsf{Hom}_{\mathrm B}\big(\mathrm V, \Res_{\mathrm B}^{\mathrm A}(\mathrm W) \big).\end{equation}
Now suppose  $\mathrm{A} = \Z_{\xx}(\HH)$, \  $\mathrm{B} = \Z_{\xx}(\GG)$, $\mathrm{V} = \Z_{\xx,\GG}^\lam$,
and $\mathrm{W} = \Z_{\xx,\HH}^\al$.     Then  
\begin{align*} \mathsf{Hom}_{\mathrm A}\big(\Ind_{\mathrm B}^{\mathrm{A}}(\Z_{\xx,\GG}^\lam), \Z_{\xx,\HH}^\al \big) &= 
\mathsf{Hom}_{\mathrm B}\big(\Z_{\xx,\GG}^\lam, \Res_{\mathrm B}^{\mathrm A}(\Z_{\xx,\HH}^\al) \big)
= 
\mathsf{Hom}_{\mathrm B}\left(\Z_{\xx,\GG}^\lam, \  \bigoplus_{\mu \in \Lambda_{\xx,\GG}} c_\al^\mu \  \Z_{\xx,\GG}^\mu\right).
\end{align*}
Taking dimensions on both sides shows that \  $\dimm \mathsf{Hom}_{\mathrm A}\big(\Ind_{\mathrm B}^{\mathrm{A}}(\Z_{\xx,\GG}^\lam), \Z_{\xx,\HH}^\al \big) = c_\al^\lam$, and thus, \break 
$\Ind_{\Z_{\xx}(\GG)}^{\Z_{\xx}(\HH)}\big(\Z_{\xx,\GG}^\lam\big) = \bigoplus_{\al \in \Lambda_{\xx,\HH}} \ c_\al^\lam \  \Z_{\xx,\HH}^\al$.  

(d)  Since $\Z_{\xx}(\HH)$ is a semisimple algebra, Wedderburn theory
tells us   $\Z_{\xx}(\HH) = \bigoplus_{\al \in \Lambda_{\xx,\HH}}  \ds_{\xx,\HH}^\al \  \Z_{\xx,\HH}^\al$,
where $\ds_{\xx,\HH}^\al = \dimm \Z_{\xx,\HH}^\al$.  
Restricting to $\Z_{\xx}(\GG)$ gives 
$$\Res_{\Z_{\xx}(\GG)}^{\Z_{\xx}(\HH)}\,\big(\Z_{\xx}(\HH)\big) =
\bigoplus_{\al \in \Lambda_{\xx,\HH}}  \ds_{\xx,\HH}^\al \ \Res_{\Z_{\xx}(\GG)}^{\Z_{\xx}(\HH)}\big(\Z_{\xx,\HH}^\al\big)
=  \bigoplus_{\lam \in \Lambda_{\xx,\GG}}  \left( \bigoplus_{\al \in \Lambda_{\xx,\HH}} c_\alpha^\lam \ \ds_{\xx,\HH}^\al \right) \Z_{\xx,\GG}^\lam,$$
by part (a).

(e) The proof here  is similar in spirit to that in parts (a) and (b).   With  $\yy$ an $\HH$-module,  suppose  $\yy= \bigoplus_{\al \in \Lambda_{\yy,\HH}} \ y_\al \HH^{\al}$,  and  assume 
$\xx\, :=\, \mathsf{Ind}_{\HH}^{\GG}(\yy) =
\bigoplus_{\lam \in \Lambda_{\xx,\GG}}  x_\lam \GG^\lam$.    Then 
\begin{align*} \xx &=  \mathsf{Ind}_{\HH}^{\GG}(\yy) =\sum_{\al \in \Lambda_{\yy,\HH}}  y_\lam \ \mathsf{Ind}_{\HH}^{\GG}(\HH^\al)
=\sum_{\al \in \Lambda_{\yy,\HH}}  y_\al \Bigg( \sum_{\lam \in \Lambda_\GG}  c_\al^\lam \ \GG^\lam \Bigg) 
= \sum_{\lam \in \Lambda_\GG} \Bigg ( \sum_{\al \in \Lambda_{\yy,\HH}} \ c_\al^\lam \ y_\al \Bigg) \GG^\lam,  \end{align*} 
so that the sum must be over $\lam \in \Lambda_{\xx,\GG}$, and 
$$\dimm \Z_{\xx,\GG}^\lam = x_\lam = \sum_{\al \in \Lambda_{\yy,\HH}}   c_\al^\lam \  y_\al
 = \sum_{\al \in \Lambda_{\yy,\HH}} \ c_\al^\lam \  \dimm \Z_{\yy,\HH}^{\al}$$ for all $\lam \in \Lambda_{\xx,\GG}$.  \end{proof}

The following  proposition  will be used in Section \ref{section:quasi} to relate multiplicities in the tensor power of the reflection module of the symmetric group  to multiplicities in tensor powers of the permutation module. Assume $\GG$ is a finite group and $\WW$ is a $\GG$-module over $\CC$.  Let  $\VV  = \GG^\bullet \oplus \WW$ be the extension of $\WW$ by the trivial $\GG$-module  $\GG^\bullet$.  Define $\Z_k(\GG)  = \End_\GG(\VV^{\otimes k})$  and  
$\QZ_k(\GG) = \End_\GG(\WW^{\otimes k})$, and let $\Lambda_{k,\GG} \subseteq \Lambda(\GG)$ (resp., $\mathsf{q}\Lambda_{k,\GG}\subseteq \Lambda(\GG)$)  index the irreducible $\GG$-modules that appear in $\VV^{\otimes k}$ (resp.,  in $\WW^{\otimes k}$) with multiplicity at least one.    Let  $\Z_k^\lambda$ (resp., $\QZ_k^\lambda$) denote the irreducible $\Z_k(\GG)$-module  (resp., $\QZ_k(\GG)$-module) indexed by  $\lambda \in \Lambda_{k,\GG}$ (resp., $\lambda\in\mathsf{q}\Lambda_{k,\GG}$).

\begin{prop} \label{prop:QuasiBinomial} With notation as in the previous paragraph,
\begin{enumerate}
\item[{\rm (a)}] 
$\displaystyle{\dimm \QZ^\lambda_{k} = \sum_{\ell = 0}^k (-1)^{k-\ell}\binom{k}{\ell} \dimm\Z^\lambda_\ell,}$

\item[{\rm (b)}] If $\WW$ is a self-dual $\GG$-module, then \ 
$\displaystyle{\dimm \QZ_{k}(\GG) =  \dimm  \QZ_{2k}^{\bullet} = \sum_{\ell = 0}^{2k} (-1)^{2k-\ell}\binom{2k}{\ell} \dimm\Z_\ell^{ \bullet}}$,  where $ \QZ_{2k}^{\bullet}$ is the irreducible $\QZ_{k}(\GG)$-module corresponding to $\GG^\bullet$; equivalently, the
space of $\GG$-invariants in $\WW^{\ot 2k}$. 
\end{enumerate}
\end{prop}

\begin{proof} Let $\chi_\VV,  \chi_\bullet, \chi_\WW$ denote the characters of $\VV, \GG^\bullet$, and $\WW$, respectively, so that $\chi_\VV =   \chi_\bullet+ \chi_\WW$. 
Then $\displaystyle{\chi_{\VV^{\otimes k}} = \chi_\VV^k = ( \chi_\bullet + \chi_\WW)^k = \sum_{\ell=0}^k \binom{k}{\ell} \chi_\WW^\ell,}$ and the binomial inverse of this statement is
\begin{align}
\label{eq:chars2}  
\chi_{\WW^{\ot k}} & =  \chi_\WW^k  = \sum_{\ell=0}^k (-1)^{k-\ell} {k \choose \ell}\ \chi_{\VV}^\ell = \sum_{\ell=0}^k (-1)^{k-\ell} {k \choose \ell}\ \chi_{\VV^{\ot \ell}}.  
\end{align} 
By Schur-Weyl duality \eqref{eq:X2G}  we have 
\begin{equation}\label{eq:sw3}
\WW^{\ot k}\  = \ \bigoplus_{\lam \in \mathsf{q}\Lambda_k} \  \qds_{k}^\lam \  \GG^\lam,
\ \quad\hbox{where \  $\qds_{k}^\lam = \dimm \QZ_{k}^\lam$.}
\end{equation}
Computing the character of \eqref{eq:sw3} and equating it with \eqref{eq:chars2} gives 
\begin{equation*}  
\chi_{\WW^{\ot k}} = \chi_\WW^k =  \sum_{\lam \in \mathsf{q}\Lambda_k} \  \qds_{k}^\lam \ \chi_\lambda \ = \
\sum_{\ell=0}^k  (-1)^{k-\ell}  {k \choose \ell} \left (\sum_{\lam \in \Lambda_\ell} \  \ds_{\ell}^\lam \ \chi_\lambda \right),
 \end{equation*}
  where $\ds_{\ell}^\lam = \dimm {\Z_{\ell}^\lam}$,  and $\chi_\lambda$ is the character of $\GG^\lambda$.
 Equating the coefficient of $\chi_\lambda$  (working in the ring of class functions on $\GG$) gives part (a).  Since $\WW$ is isomorphic to its dual as a $\GG$-module, part (b) is the special case of part (a) with $\lambda = \bullet$ (the index of the trivial module):
 $$
\dimm \QZ_{k}(\GG) = \dimm  \QZ_{2k}^{\bullet} = \sum_{\ell = 0}^{2k} (-1)^{2k-\ell}\binom{2k}{\ell} \dimm\Z_\ell^{ \bullet}. \hspace{1.5cm} \qedhere
$$
\end{proof}

 \begin{section}{Irreducible modules for symmetric and alternating groups and \\ their centralizer algebras}\label{sec:Sn}\end{section} The irreducible $\S_n$-modules are labeled by
partitions of $n$, so that  $\Lambda_{\S_n} = \{\lam \mid \lam \vdash n\}$.   When writing  $\lam = [\lam_1, \dots, \lam_n] \vdash n$,
we always assume that the parts of the partition $\lam$  are arranged so that $\lam_1 \geq \lam_2 \ge \ldots \ge \lam_n \ge 0$, and
$|\lam| = n$ (the sum of the parts). 
 We identify a partition with its Young diagram, so for $\lam =[6,4,3,2^2] \vdash 17$,
 we have  

\begin{center}{$\lambda= \begin{array}{c} 
{\begin{tikzpicture}[scale=.35,line width=5pt] 
\tikzstyle{Pedge} = [draw,line width=.7pt,-,black]    
\foreach \i in {1,...,7} 
{\path (\i,1) coordinate (T\i); 
\path (\i,0) coordinate (B\i);
\foreach \i in {1,...,5}  \path (\i,-1) coordinate (S\i);
\foreach \i in {1,...,4}  \path (\i,-2) coordinate (R\i);  
\foreach \i in {1,...,3}  \path (\i,-3) coordinate (Q\i);  
\foreach \i in {1,...,3}  \path (\i,-4) coordinate (P\i);} 
\filldraw[fill= white!10,draw=white!10,line width=8pt]   (T1) -- (T7) -- (B7) -- (B1) -- (T1);
\filldraw[fill= white!10,draw=white!10,line width=8pt]   (B1) -- (B5) -- (S5) -- (S1) -- (B1);
\filldraw[fill= white!10,draw=white!10,line width=8pt]   (S1) -- (S4) -- (R4) -- (R1) -- (S1);
\filldraw[fill= white!10,draw=white!10,line width=8pt]   (R1) -- (R3) -- (Q3) -- (Q1) -- (R1);
\filldraw[fill= white!10,draw=white!10,line width=8pt]   (Q1) -- (Q3) -- (P3) -- (P1) -- (Q1);
\filldraw[fill= gray!60,draw=gray!60,line width=.35pt]   (B2) -- (B3) -- (S3) -- (S2) -- (B2);
\path (T1) edge[Pedge] (P1);
\path(T2)  edge[Pedge] (P2);
\path (T3)  edge[Pedge] (P3);
\path (T4)  edge[Pedge] (R4);
\path (T5)  edge[Pedge] (S5);
\path (T6)  edge[Pedge] (B6);
\path (T7)  edge[Pedge] (B7);
\path(T1) edge[Pedge] (T7);
\path(B1) edge[Pedge] (B7);
\path(S1) edge[Pedge] (S5);
\path(R1) edge[Pedge] (R4); 
\path(Q1) edge[Pedge] (Q3); 
\path(P1) edge[Pedge] (P3);  
\end{tikzpicture}}\end{array}.$}\end{center} 
The  \emph{hook length}  $h(b)$ of a box $b$ in the diagram  is 
 1 plus  the number of boxes below  $b$ in the same column   
plus the number of boxes to the right of $b$ in the same row, and  $h(b) = 1+3+2 = 6$ for the shaded box above.   
The dimension of the irreducible $\S_n$-module $\S_n^\lam$, which we denote $f^\lam$,  can be easily computed by the well-known  hook-length formula   
\begin{equation}\label{eq:hook} f^\lambda = \frac{n!}{\prod_{b \in \lambda} h(b)},\end{equation}
where the denominator is the product of the hook lengths as $b$ ranges over the boxes in the Young diagram of $\lambda$. 
This is  equal to the number of standard Young tableaux of shape $\lam$, where a standard Young tableau $T$ is a filling of the boxes in the Young diagram of $\lam$ with the numbers $\{1, \ldots,n\}$ such that the entries increase in every row from left to right  and in every column from top to bottom.   

The restriction and induction rules for  irreducible symmetric group modules $\S_n^\lambda$ are well known (see for example \cite[Thm.~2.43]{JK}):
\begin{equation}\label{eq:RI}  \mathsf{Res}^{\S_n}_{\S_{n-1}}(\S_n^\lambda)  = \bigoplus_{\mu = \lambda-\square} \S_{n-1}^\mu
\qquad  \mathsf{Ind}^{\S_{n+1}}_{\S_{n}}(\S_n^\lambda)  = \bigoplus_{\kappa = \lambda+\square} \S_{n+1}^\kappa, \end{equation}
where the first sum is over all partitions $\mu$ of $n-1$ obtained from $\lambda$ by removing a box from the end of a row of the diagram of
$\lambda$, and
the second sum is over all partitions $\kappa$ of $n+1$ obtained by adding a box to the end of a row of $\lambda$.  

Assume $\lam =[\lam_1, \dots, \lam_{\ell(\lam)}]$ is a partition of $n$ and $\gamma = [\gamma_1,\dots, \gamma_n]$ is a weak composition of $n$. 
The \emph{Kostka number} $\KK_{\lam,\gamma}$ counts the number of semistandard tableaux $T$ of shape $\lambda$ and type $\gamma$,
where $T$ is a filling of the boxes of the Young diagram of $\lam$ with numbers from $\{1,\dots, n\}$ such that $j$ occurs $\gamma_j$
times, and the entries of $T$ weakly increase across the rows from left to right and strictly increase down the columns.    If $\gamma
= [\gamma_1,\gamma_2,\dots, \gamma_{\ell(\gamma)}]$
is a partition, then $\KK_{\lam,\gamma} = 0$, unless $\lam \geq \gamma$ in the dominance order, which is to say  
that for 
the first part where $\lam$ and $\gamma$
 differ  $\lam_j > \gamma_j$.    Assume  $\perm^\gamma$ is the \emph{permutation module}  obtained from inducing the trivial module for
 the Young subgroup $\S_{\gamma_1} \times \S_{\gamma_2} \times \cdots \times \S_{\gamma_{\ell(\gamma)}}$ to $\S_n$. 
 The irreducible $\S_n$-module $\S_n^\lam$ occurs with multiplicity $\KK_{\lam,\gamma}$
in the decomposition of   $\perm^\gamma$ into irreducible $\S_n$-summands.

For  $\lam  \vdash n$, let $\lam^{\ast
}$ be the conjugate (transpose) 
partition.  Since $\S_n^{\lam^{\ast
}} \cong \S_n^{[1^n]} \ot \S_n^{\lam}$,  where $ \S_n^{[1^n]}$  is the one-dimensional irreducible
$\S_n$-module indexed by the partition of $n$ into $n$ parts of size one, which is the sign representation,  
$\S_n^{\lam^{\ast}} \cong \S_n^\lam$  as $\A_n$-modules.   
Thus, we may assume that  $\lam \geq \lam^{\ast
}$ in the dominance order.   Then by Clifford theory,  it is known that  
\begin{equation}\label{eq:Snrest}  \mathsf{Res}^{\S_n}_{\A_n}(\S_n^\lam) =  \begin{cases} \A_n^{\lam} \cong \mathsf{Res}^{\S_n}_{\A_n}(\S_n^{\lam^{\ast}})&  \quad  \hbox{\rm if} \ \ \lam \neq  \lam^{\ast
}, \\
\A_n^{\lam^+} \oplus  \A_n^{\lam^-} &  \quad  \hbox{\rm if} \ \ \lam= \lam^{\ast}, 
\end{cases} \end{equation} 
where in the first case,  $\A_n^\lam$ is irreducible as an $\A_n$-module;  while  in the second case,  $\S_n^\lam$ decomposes into 
the direct sum of two irreducible $\A_n$-modules,  $\S_n^\lam =\A_n^{\lam^+} \oplus \A_n^{\lam^-}$,  such that
$\dimm \A_n^{\lam^+} = \dimm \A_n^{\lam^-} =  \frac{1}{2} \dimm \S_n^{\lam} =  \frac{1}{2}f^\lam$.   Moreover, 
$$\Lambda_{\A_n}  = \{ \lam \mid \lam  \vdash n, \ \lam>\lam^{\ast} \} \cup \{ \lam^{\pm} \mid \lam \vdash n, \ \lam = \lam^{\ast}  \}.$$  

The restriction rules for alternating groups are the following (see \cite[Thm.\,6.1]{R}, or \cite{Mb} which surveys how to derive these
rules using Mackey's theorem and Clifford theory):  
\begin{align}\begin{split}\label{eq:Aind-res}  
\hspace{-.7cm}  \mathsf{Res}^{\A_n}_{\A_{n-1}}(\A_n^\lam) & =  
 \left(\bigoplus_{{\mu = \lam-\square} \atop {\mu > \mu^\ast}} \A_{n-1}^\mu  \right)  
 \oplus
 \left( \bigoplus_{{\mu = \lam-\square} \atop {\mu = \mu^\ast}} (\A_{n-1}^{\mu^+} \oplus \A_{n-1}^{\mu^-})\right) \ \, \hbox{\rm if \,  $\lam > \lam^\ast$},  \\
  \mathsf{Res}^{\A_n}_{\A_{n-1}}(\A_n^{\lam^{\pm}}) & =   \left(\bigoplus_{{\mu = \lam-\square} \atop {\mu > \mu^\ast
}} \A_{n-1}^\mu\right)
 \oplus
\left( \bigoplus_{{\mu = \lam-\square} \atop {\mu = \mu^\ast}} \A_{n-1}^{\mu^\pm} \right)  \qquad  \hskip.37in \hbox{\rm if \  \ 
 $\lam = \lam^\ast$.} 
  \end{split} \end{align}  
   
Now let $\modu$ be the $n$-dimensional  permutation module for $\S_n$, and set  $\xx = \modu^{\ot k}$ for $k \geq 0$ in applying the results of Section \ref{sec:restrictioninduction},  where $\modu^{\ot 0} =\S_n^{[n]}$ (the trivial $\S_n$-module).   Since $\modu$ will be fixed throughout, it
convenient here to  adopt the  shorthand notation in Table \ref{table1}. For all $k \in \ZZ_{\ge 0}$,   $\Lambda_{k,\S_n}$  (resp. $\Lambda_{k,\A_n}$) 
is the set of indices for the irreducible $\Z_k(n)$-summands   
(resp. $\WZ_k(n)$-summands) in $\modu^{\ot k}$ with multiplicity at least one;  similarly      
 $\Lambda_{k,\S_{n-1}}$  (resp. $\Lambda_{k,\A_{n-1}}$) 
 is the set of indices for the irreducible $\Z_{k+\half}(n)$-summands   
(resp. $\WZ_{k+\half}(n)$-summands) in $\modu^{\ot k}$ with multiplicity at least one.

\begin{table}[h!]
\centering
\begin{tabular}[t]{|c|c|}
\hline
        centralizer algebra  &  irreducible modules \\ \hline \hline  
 $\Z_k(n): = \End_{\S_n} (\modu^{\ot k})\phantom{\bigg\vert}$ \!\! &  $\Z_{k,n}^\lam, \ \  \lam \vdash n$, \ \ $\lam \in \Lambda_{k,\S_n} \subseteq \Lambda_{\S_n}$  \\ \hline
\!\! $\Z_{k+\half}(n): = \End_{\S_{n-1}} (\modu^{\ot k})\phantom{\bigg\vert}$\!\!  &  $\Z_{k+\half,n}^\mu, \ \  \mu \vdash n-1$, \  \ $\mu \in \Lambda_{k,\S_{n-1}} \subseteq \Lambda_{\S_{n-1}}$ \\ \hline 
  $\widehat{\Z}_k(n): = \End_{\A_n} (\modu^{\ot k})\phantom{\bigg\vert}$ \!\! &  $\widehat\Z_{k,n}^\lam, \ \  \lam \vdash n$, \ $\lam > \lam^{\ast
}$,
\  \ $\lam  \in \Lambda_{k,\A_{n}}  \subseteq \Lambda_{\A_n}$\!\!  \\  
  & $\widehat \Z_{k,n}^{\lam^\pm}, \ \  \lam \vdash n$, \ $\lam = \lam^{\ast
}$, \  \ $\lam^\pm \in \Lambda_{k,\A_{n}}  \subseteq \Lambda_{\A_n}\phantom{\bigg\vert}$  \!\!\\  \hline
\!\!  $\widehat{\Z}_{k+\half}(n): = \End_{\A_{n-1}} (\modu^{\ot k})$  \!\! &\!\!  $\widehat\Z_{k+\half,n}^\mu, \ \  \mu \vdash n-1$, \ $\mu> \mu^{\ast}$,
   \  \ $\mu  \in \Lambda_{k+\half,\A_{n-1}}  \subseteq \Lambda_{\A_{n-1}}\phantom{\bigg\vert}$\!\!   \!\!\\  
  &$\widehat \Z_{k+\half,n}^{\mu^\pm}, \ \  \mu \vdash n$, \, \ $\mu = \mu^{\ast}$, \ \ $\mu^\pm  \in \Lambda_{k+\half,\A_{n-1}}
  \subseteq \Lambda_{\A_{n-1}}$  \phantom{\bigg\vert} \!\!\!\!\\  \hline
	\end{tabular}
	\caption{Notation for the centralizer algebras and modules associated with the tensor  product $\mathsf{M}_n^{\ot k}$ of the permutation module $\mathsf{M}_n \cong  \mathsf{S}_n^{[n]}\oplus\mathsf{S}_n^{[n-1,1]}$ of $\S_n$ and its restriction to 
	 $\S_{n-1}$, $\A_n$, and $\A_{n-1}$. \label{table1}}
	\end{table}

Theorem \ref{T:res-ind}(b) together with \eqref{eq:Snrest} imply the following:   \medskip

\begin{prop}\label{P:zdim} Assume $\lam \vdash n$, $\lam \in \Lambda_{k,\A_n}$,  and $\lam \geq \lam^{\ast}$. Then 
\begin{align}\begin{split}\label{eq:branch} \dimm \WZ_{k,n}^{\lam}   &= \dimm \Z_{k,n}^\lam + \dimm\Z_{k,n}^{\lam^{\ast
}}, \quad \quad \ \  \hbox{if \  $\lam > \lam^{\ast}$,} \\ 
\dimm \WZ_{k,n}^{\lambda^+}  &= \dimm \WZ_{k,n}^{\lambda^-}  =\dimm \Z_{k,n}^{\lambda}, \ \ \qquad \hbox{if \ $\lam = \lam^{\ast
}$.}
\end{split} \end{align} 
\end{prop} \medskip

\begin{example}   For $\S_4$,  we have $\mathsf{M}_4^{\ot 3} = 5 \S_4^{[4]} \oplus  10 \S_4^{[3,1]}\oplus 5 \S_4^{[2^2]} \oplus 6 \S_4^{[2,1^2]} 
\oplus \S_4^{[1^4]}$,   and  for $\A_4$, \ $\mathsf{M}_4^{\ot 3}  = 6 \A_4^{[4]} \oplus  16 \A_4^{[3,1]}\oplus 5 \A_4^{[2^2]^+}\oplus 5 \A_4^{[2^2]^-}$, as can be seen in row $\ell=3$  of Figures \ref{fig:Sbratteli} and \ref{fig:Abratteli},
where  
\begin{align*} \dimm \WZ_{3,4}^{[4]} &= \dimm \Z_{3,4}^{[4]}   + \dimm \Z_{3,4}^{[1^4]}  = 5+1 = 6, \\
\dimm \WZ_{3,4}^{[3,1]}  &= \dimm \Z_{3,4}^{[3,1]}  + \dimm \Z_{3.4}^{[2,1^2]}  = 10+6 = 16, \\
\dimm \WZ_{3,4}^{[2^2]^{\pm}}  & = \dimm \Z_{3,4}^{[2^2]}= 5. \end{align*}   \end{example}  
  
\section{Bratteli diagrams} \label{sec:Bratteli}

Let $(\GG,\HH)$ be a pair consisting of a finite group $\GG$ and a subgroup $\HH \subseteq \GG$. As in Section \ref{sec:restrictioninduction}, let 
$\{\GG^\lam\}_{\lam \in \Lambda_\GG}$ and $\{\HH^\alpha\}_{\alpha\in \Lambda_\HH}$ be the irreducible modules of $\GG$ and $\HH$ over $\CC$ with 
restriction and induction rules given by
\begin{equation}\label{eq:resind} 
\mathsf{Res}_\HH^\GG(\GG^\lam) = \bigoplus_{\alpha \in \Lambda_\HH} c_\alpha^\lam \  \HH^\al \qquad\hbox{and}\qquad
 \mathsf{Ind}_{\HH}^{\GG}(\HH^\alpha) =  \bigoplus_{\lam \in \Lambda_\GG } c_\alpha^\lam \ \GG^\lam.
 \end{equation}

Let   $\UU^0 =\GG^\bullet$, the trivial $\GG$-module, and assume for $k \in  \ZZ_{\ge 0}$ that  the $\GG$-module $\UU^{k}$ has been defined.    
Let $\UU^{k + \half}$ be the $\HH$-module defined by $\UU^{k+\half} := \Res_{\HH}^{\GG}(\UU^k)$,  and 
then let  $\UU^{k+1}$ be  the $\GG$-module specified by  $\UU^{k+1} := \Ind_{\HH}^{\GG}(\UU^{k+\half})$.  
In this way,  $\UU^\ell$ is defined inductively for all $\ell \in \half \ZZ_{\ge 0}$,   and $\UU^k = \left( \Ind_\HH^\GG  \Res_\HH^\GG\right)^k(\UU^0)$
for all $k \in \ZZ_{\ge 0}$.    The module  $\VV := \mathsf{Ind}_\HH^\GG(\mathsf{Res}_\HH^\GG(\UU^0)) = \UU^1$  is isomorphic to $\GG/\HH$ as a $\GG$-module, where $\GG$ acts on the left cosets
of $\GG/\HH$ by multiplication.

For a $\GG$-module $\xx$ and an $\HH$-module  $\yy$, the ``tensor identity"  says that
 $\Ind^{\GG}_{\HH}(\Res^{\GG}_{\HH}(\xx) \otimes \yy) \cong \xx \otimes \Ind^{\GG}_{\HH}(\yy)$ (see for example \cite[(3.18)]{HR}  for an explicit isomorphism).   Hence, when $\xx= \UU^{k}$ and $\yy = \mathsf{Res}_\HH^\GG(\UU^0)$, this gives 
\begin{equation}\label{eq:tens} \Ind^{\GG}_{\HH}(\Res^{\GG}_{\HH}(\UU^k)) \cong \Ind^{\GG}_{\HH}(\Res^{\GG}_{\HH}(\UU^k)\otimes \mathsf{Res}_\HH^\GG(\UU^0)) \cong \UU^k \otimes \Ind^{\GG}_{\HH}( \mathsf{Res}_\HH^\GG(\UU^0)) = \UU^{k} \otimes \VV.\end{equation}
   By induction, we have the following isomorphisms for all $k \in \ZZ_{\ge 0}$:  
\begin{equation}
 \VV^{\otimes k} \cong \UU^k\quad (\text{as $\GG$-modules)} \qquad\text{and}\qquad  \Res^\GG_\HH(\VV^{\otimes k}) \cong \UU^{k + \frac{1}{2}}\quad (\text{as $\HH$-modules}).
\end{equation}
It follows that there are centralizer algebras isomorphisms:
\begin{align}\begin{split}\label{eq:tensisos}  \Z_k(\GG) &:=\End_\GG(\VV^{\otimes k}) \cong \End_\GG(\UU^k), \\ 
\Z_{k+\frac{1}{2}}(\HH)& :=\End_\HH(\Res^\GG_\HH(\VV^{\otimes k})) \cong \End_\HH(\UU^{k+\frac{1}{2}}).
\end{split}\end{align}    
   
Suppose for $k \in \ZZ_{\ge 0}$ that 
\begin{itemize}
\item $\Lambda_{k,\GG} \subseteq \Lambda_\GG$ indexes the irreducible $\GG$-modules, and hence also the irreducible $\Z_k(\GG)$-modules,  in $\UU^{k} \cong \VV^{\otimes k}$;

\item  $\Lambda_{k+\frac{1}{2},\HH}\subseteq \Lambda_\HH$  indexes the irreducible $\HH$-modules, and hence also the irreducible $\Z_{k+\frac{1}{2}}(\HH)$-modues,  in $\UU^{k + \frac{1}{2}}\cong\Res^\GG_\HH(\VV^{\otimes k})$.
\end{itemize}
The \emph{restriction-induction Bratteli diagram} for the pair $(\GG,\HH)$  is an infinite, rooted tree $\mathcal{B}(\GG,\HH)$ whose vertices are organized into rows labeled by half integers $\ell$  in $\half\ZZ_{\ge 0}$.  For $\ell =k \in \ZZ_{\ge 0}$, the vertices on row
$k$ are the elements of $\Lambda_{k,\GG}$, and the vertices on row $\ell = k+\half$ are the elements of $\Lambda_{k+\half,\HH}$. 
The vertex on row $0$ is the root, the label of the trivial $\GG$-module, and the vertex on row $\half$ is  the label
of the trivial $\HH$-module.  For the pair $(\S_n,\S_{n-1})$ (or $(\A_n,\A_{n-1})$),  the labels on rows $0$ and $\frac{1}{2}$ are the partitions $[n], [n-1]$ having just one part.

 The edges of $\mathcal{B}(\GG,\HH)$ are given by drawing $c^\lambda_\alpha$ edges from $\lambda \in \Lambda_{k,\GG}$ to $\alpha \in \Lambda_{k+\frac{1}{2},\HH}$, where  $c^\lambda_\alpha$ is as in \eqref{eq:resind}. 
Similarly, there are  $c_\beta^\kappa$  edges 
 from $\beta \in \Lambda_{k+\frac{1}{2},\HH}$ to $\kappa \in \Lambda_{k+1,\GG}$.  The Bratteli diagram is constructed in such a way that 
 \begin{itemize}
 
\item the number of paths from the root at level 0 to $\lambda \in\Lambda_{k,\GG}$ equals the multipicity of $\GG^\lam$ in $\UU^{k}\cong\VV^{\otimes k}$ and thus also equals the dimension of the irreducible $\Z_k(\GG)$-module  $\Z_{\UU^k,\GG}^\lam$ (these numbers are computed in Pascal-triangle-like fashion and are placed below each vertex);

\item the number of paths  from the root at level 0 to $\alpha \in\Lambda_{k+\frac{1}{2},\HH}$ equals the multipicity of $\HH^\alpha$ in $\UU^{k + \frac{1}{2}}$ and thus also equals the dimension of
the $\Z_{k+\half}(\HH)$-module  $\Z_{\UU^{k+\half},\HH}^\alpha$ (and is indicated beneath each vertex);

\item the sum of the squares of the labels on row $k$ (resp. row $k+\half$)  equals $\dimm \Z_k(\GG)$ (resp. $\dimm \Z_{k+\half}(\HH)$).
\end{itemize}

When  $(\GG,\HH) = (\S_n,\S_{n-1})$  or  when $(\GG,\HH) = (\A_n, \A_{n-1})$, it is well known (and easy to verify) that the permutation module satisfies  $\modu  \cong \UU^1 = \Ind^{\S_n}_{\S_{n-1}}(\Res^{\S_n}_{\S_{n-1}}(\S_n^{[n]}))$.   Then \eqref{eq:tensisos} implies
there are  partition algebra surjections
as $\P_k(n)\to \Z_k(n) =  \End_{\S_n}(\modu^{\otimes k})\cong\End_{\S_n}(\UU^k)$ and $\P_{k+\frac{1}{2}}(n)\to \Z_{k+\frac{1}{2}}(n) =  \End_{\S_{n-1}}(\modu^{\otimes k})\cong\End_{\S_{n-1}}(\UU^k)$.    Using the restriction/induction rules for $\S_n$ in  \eqref{eq:RI} 
and for $\A_n$ in  \eqref{eq:Aind-res},  we construct  the Bratteli diagram for $(\S_4,\S_3)$  (see  Figure \ref{fig:Sbratteli}) and
for $(\A_4,\A_3)$  (see  Figure \ref{fig:Abratteli}).  In  Appendices \ref{sec:BratteliS6S5} and \ref{sec:BratteliA6A5},  we construct the Bratteli diagrams for $(\S_6,\S_5)$ and $(\A_6,\A_5)$. 

\Yboxdim{8pt}  
\Ylinethick{.6pt}  
\begin{figure}[h!]
$$
\begin{tikzpicture}[line width=.5pt,xscale=0.15,yscale=0.25]  
\path (-10,0)  node[anchor=west]  {$\bsym{\ell=0}$};
\path (0,0)  node[anchor=west] (S4-0) {\yng(4)};
\draw (S4-0) node[below=4pt, black] {\nmm 1};
\path (-10,-5)  node[anchor=west]  {$\bsym{\ell=\frac{1}{2}}$};
\path (10,-5)  node[anchor=west] (S3-1) {\yng(3)};
\draw (S3-1) node[below=4pt,black] {\nmm 1};
\path (-10,-10)  node[anchor=west]  {$\bsym{\ell=1}$};
\path (0,-10)  node[anchor=west] (S4-2) {\yng(4)};
\path (19,-10)  node[anchor=west] (S31-2) {\yng(3,1)};
\draw (S4-2) node[below=4pt,black] {\nmm 1};
\draw (S31-2) node[below=6pt, right=-2pt,black] {\nmm 1};
\path (-10,-15)  node[anchor=west]  {$\bsym{\ell=\frac{3}{2}}$};
\path (10,-15)  node[anchor=west] (S3-3) {\yng(3)};
\path (29,-15)  node[anchor=west] (S21-3) {\yng(2,1)};
\draw (S3-3) node[below=4pt, black] {\nmm 2};
\draw (S21-3) node[below=6pt,right=2pt,black] {\nmm 1};
\path (-10,-20)  node[anchor=west]  {$\bsym{\ell=2}$};
\path (0,-20)  node[anchor=west] (S4-4) {\yng(4)};
\path (19,-20)  node[anchor=west] (S31-4) {\yng(3,1)};
\path (29,-20)  node[anchor=west] (S22-4) {\yng(2,2)};
\path (38,-20)  node[anchor=west] (S211-4) {\yng(2,1,1)};
\draw (S4-4) node[below=4pt, black] {\nmm 2};
\draw (S31-4) node[below=6pt, right=-2pt,black] {\nmm 3};
\draw (S22-4) node[below=4pt, right=8pt, black] {\nmm 1};
\draw (S211-4) node[below=4pt, right=2pt, black] {\nmm 1};
\path (-10,-25)  node[anchor=west]  {$\bsym{\ell=\frac{5}{2}}$};
\path (10,-25)  node[anchor=west] (S3-5) {\yng(3)};
\path (29,-25)  node[anchor=west] (S21-5) {\yng(2,1)};
\path (48,-25)  node[anchor=west] (S111-5) {\yng(1,1,1)};
\draw (S3-5) node[below=4pt,black] {\nmm 5};
\draw (S21-5) node[below=6pt,right=2pt,black] {\nmm 5};
\draw (S111-5) node[below=4pt, right=6pt,black] {\nmm 1};
\path (-10,-30)  node[anchor=west]  {$\bsym{\ell=3}$};
\path (0,-30)  node[anchor=west] (S4-6) {\yng(4)};
\path (19,-30)  node[anchor=west] (S31-6) {\yng(3,1)};
\path (29,-30)  node[anchor=west] (S22-6) {\yng(2,2)};
\path (38,-30)  node[anchor=west] (S211-6) {\yng(2,1,1)};
\path (57,-30)  node[anchor=west] (S1111-6) {\yng(1,1,1,1)};
\draw (S4-6) node[below=4pt,black] {\nmm 5};
\draw (S31-6) node[below=6pt, right=-2pt,black] {\nmm 10};
\draw (S22-6) node[below=4pt, right=8pt, black] {\nmm 5};
\draw (S211-6) node[below=4pt, right=2pt, black]  {\nmm 6};
\draw (S1111-6) node[below=4pt, right=6pt,black] {\nmm 1};
\path (-10,-35)  node[anchor=west]  {$\bsym{\ell=\frac{7}{2}}$};
\path (10,-35)  node[anchor=west] (S3-7) {\yng(3)};
\path (29,-35)  node[anchor=west] (S21-7) {\yng(2,1)};
\path (48,-35)  node[anchor=west] (S111-7) {\yng(1,1,1)};
\draw (S3-7) node[below=4pt,black] {\nmm 15};
\draw (S21-7) node[below=6pt,right=2pt,black] {\nmm 21};
\draw (S111-7) node[below=4pt, right=6pt,black] {\nmm 7};
\path (-10,-40)  node[anchor=west]  {$\bsym{\ell=4}$};
\path (0,-40)  node[anchor=west] (S4-8) {\yng(4)};
\path (19,-40)  node[anchor=west] (S31-8) {\yng(3,1)};
\path (29,-40)  node[anchor=west] (S22-8) {\yng(2,2)};
\path (38,-40)  node[anchor=west] (S211-8) {\yng(2,1,1)};
\path (57,-40)  node[anchor=west] (S1111-8) {\yng(1,1,1,1)};
\draw (S4-8) node[below=4pt,black] {\nmm 15};
\draw (S31-8) node[below=6pt, right=-2pt,black] {\nmm 36};
\draw (S22-8) node[below=4pt, right=8pt, black] {\nmm 21};
\draw (S211-8) node[below=4pt, right=2pt, black] {\nmm 28};
\draw (S1111-8) node[below=4pt, right=6pt,black] {\nmm 7};
\draw (75,0) node[anchor=east,black] {$~~{\mathbf{1}}$};
\draw (75,-5) node[anchor=east,black] {$~~{\mathbf{1}}$};
\draw (75,-10) node[anchor=east,black] {$~~{\mathbf{2}}$};
\draw (75,-15) node[anchor=east,black] {$~~{\mathbf{5}}$};
\draw (75,-20) node[anchor=east,black] {$~~{\mathbf{15}}$};
\draw (75,-25) node[anchor=east,black] {$~~{\mathbf{51}}$};
\draw (75,-30) node[anchor=east,black] {$~~{\mathbf{187}}$};
\draw (75,-35) node[anchor=east,black] {$~~{\mathbf{715}}$};
\draw (75,-40) node[anchor=east,black] {$~~{\mathbf{2795}}$};
\path  (S4-0) edge[black,thick] (S3-1);
\path  (S4-2) edge[black,thick] (S3-1);
\path  (S31-2) edge[black,thick] (S3-1);
\path  (S4-2) edge[black,thick] (S3-3);
\path  (S31-2) edge[black,thick] (S3-3);
\path  (S31-2) edge[black,thick] (S21-3);
\path  (S4-4) edge[black,thick] (S3-3);
\path  (S31-4) edge[black,thick] (S3-3);
\path  (S31-4) edge[black,thick] (S21-3);
\path  (S22-4) edge[black,thick] (S21-3);
\path  (S211-4) edge[black,thick] (S21-3);
\path  (S4-4) edge[black,thick] (S3-5);
\path  (S31-4) edge[black,thick] (S3-5);
\path  (S31-4) edge[black,thick] (S21-5);
\path  (S22-4) edge[black,thick] (S21-5);
\path  (S211-4) edge[black,thick] (S21-5);
\path  (S211-4) edge[black,thick] (S111-5);
\path  (S4-6) edge[black,thick] (S3-5);
\path  (S31-6) edge[black,thick] (S3-5);
\path  (S31-6) edge[black,thick] (S21-5);
\path  (S22-6) edge[black,thick] (S21-5);
\path  (S211-6) edge[black,thick] (S21-5);
\path  (S211-6) edge[black,thick] (S111-5);
\path  (S1111-6) edge[black,thick] (S111-5);
\path  (S4-6) edge[black,thick] (S3-7);
\path  (S31-6) edge[black,thick] (S3-7);
\path  (S31-6) edge[black,thick] (S21-7);
\path  (S22-6) edge[black,thick] (S21-7);
\path  (S211-6) edge[black,thick] (S21-7);
\path  (S211-6) edge[black,thick] (S111-7);
\path  (S1111-6) edge[black,thick] (S111-7);
\path  (S4-8) edge[black,thick] (S3-7);
\path  (S31-8) edge[black,thick] (S3-7);
\path  (S31-8) edge[black,thick] (S21-7);
\path  (S22-8) edge[black,thick] (S21-7);
\path  (S211-8) edge[black,thick] (S21-7);
\path  (S211-8) edge[black,thick] (S111-7);
\path  (S211-8) edge[black,thick] (S111-7);
\path  (S1111-8) edge[black,thick] (S111-7);
\end{tikzpicture}
$$
\caption{Levels  \ $\ell = 0,\half,1,\ldots,\frac{7}{2},4$ \  of the Bratteli diagram for the pair $(\S_4, \S_3)$.\label{fig:Sbratteli}}
\end{figure}

\Yboxdim{8pt}  
\Ylinethick{.6pt}  
\begin{figure}[h!]
$$
\begin{tikzpicture}[line width=.5pt,xscale=0.18,yscale=0.25]
\path (-10,0)  node[anchor=west]  {$\bsym{\ell=0}$};
\path (0,0)  node[anchor=west] (S4-0) {\yng(4)};
\draw (S4-0) node[below=12pt,right=0pt,black] {\nmm 1};
\path (-10,-5)  node[anchor=west]  {$\bsym{\ell=\frac{1}{2}}$};
\path (.8,-5)  node[anchor=west] (S3-1) {\yng(3)};
\draw (S3-1) node[below=12pt,right=0pt,black] {\nmm 1};
\path (-10,-10)  node[anchor=west]  {$\bsym{\ell=1}$};
\path (0,-10)  node[anchor=west] (S4-2) {\yng(4)};
\path (12,-10)  node[anchor=west] (S31-2) {\yng(3,1)};
\draw (S4-2) node[below=12pt,right=0pt,black] {\nmm 1};
\draw (S31-2) node[below=6pt, right=-2pt,black] {\nmm 1};
\path (-10,-15)  node[anchor=west]  {$\bsym{\ell=\frac{3}{2}}$};
\path (.8,-15)  node[anchor=west] (S3-3) {\yng(3)};
\path (12.25,-15)  node[anchor=west] (S21p-3) {$\yng(2,1)^{\bsym{+}}$};
\path (23,-15)  node[anchor=west] (S21m-3) {$\yng(2,1)^{\bsym{-}}$};
\draw (S3-3) node[below=12pt,right=0pt,black] {\nmm 2};
\draw (S21p-3) node[below=6pt,right=2pt,black] {\nmm 1};
\draw (S21m-3) node[below=6pt,right=2pt,black] {\nmm 1};
\path (-10,-20)  node[anchor=west]  {$\bsym{\ell=2}$};
\path (0,-20)  node[anchor=west] (S4-4) {\yng(4)};
\path (12,-20)  node[anchor=west] (S31-4) {\yng(3,1)};
\path (23,-20)  node[anchor=west] (S22p-4) {$\yng(2,2)^{\bsym{+}}$};
\path (33,-20)  node[anchor=west] (S22m-4) {$\yng(2,2)^{\bsym{-}}$};
\draw (S4-4) node[below=12pt,right=0pt,black]{\nmm 2};
\draw (S31-4) node[below=6pt, right=-2pt,black] {\nmm 4};
\draw (S22p-4) node[below=4pt, right=8pt, black] {\nmm 1};
\draw (S22m-4) node[below=4pt, right=8pt, black] {\nmm 1};
\path (-10,-25)  node[anchor=west]  {$\bsym{\ell=\frac{5}{2}}$};
\path (.8,-25)  node[anchor=west] (S3-5) {\yng(3)};
\path (12.25,-25)  node[anchor=west] (S21p-5) {$\yng(2,1)^{\bsym{+}}$};
\path (23,-25)  node[anchor=west] (S21m-5) {$\yng(2,1)^{\bsym{-}}$};
\draw (S3-5) node[below=12pt,right=0pt,black] {\nmm 6};
\draw (S21p-5) node[below=6pt,right=2pt,black] {\nmm 5};
\draw (S21m-5) node[below=6pt,right=2pt,black] {\nmm 5};
\path (-10,-30)  node[anchor=west]  {$\bsym{\ell=3}$};
\path (0,-30)  node[anchor=west] (S4-6) {\yng(4)};
\path (12,-30)  node[anchor=west] (S31-6) {\yng(3,1)};
\path (23,-30)  node[anchor=west] (S22p-6) {$\yng(2,2)^{\bsym{+}}$};
\path (33,-30)  node[anchor=west] (S22m-6) {$\yng(2,2)^{\bsym{-}}$};
\draw (S4-6) node[below=12pt,right=0pt,black] {\nmm 6};
\draw (S31-6) node[below=6pt, right=-2pt,black] {\nmm 16};
\draw (S22p-6) node[below=4pt, right=8pt, black] {\nmm 5};
\draw (S22m-6) node[below=4pt, right=8pt, black] {\nmm 5};
\path (-10,-35)  node[anchor=west]  {$\bsym{\ell=\frac{7}{2}}$};
\path (.8,-35)  node[anchor=west] (S3-7) {\yng(3)};
\path (12.25,-35)  node[anchor=west] (S21p-7) {$\yng(2,1)^{\bsym{+}}$};
\path (25,-35)  node[anchor=west] (S21m-7) {$\yng(2,1)^{\bsym{-}}$};
\draw (S3-7) node[below=6pt,right=10pt,black] {\nmm 22};
\draw (S21p-7) node[below=6pt,right=2pt,black] {\nmm 21};
\draw (S21m-7) node[below=6pt,right=2pt,black] {\nmm 21};
\draw (50,0) node[anchor=east,black] {$~~{\mathbf{1}}$};
\draw (50,-5) node[anchor=east,black] {$~~{\mathbf{1}}$};
\draw (50,-10) node[anchor=east,black] {$~~{\mathbf{2}}$};
\draw (50,-15) node[anchor=east,black] {$~~{\mathbf{6}}$};
\draw (50,-20) node[anchor=east,black] {$~~{\mathbf{22}}$};
\draw (50,-25) node[anchor=east,black] {$~~{\mathbf{86}}$};
\draw (50,-30) node[anchor=east,black] {$~~{\mathbf{342}}$};
\draw (50,-35) node[anchor=east,black] {$~~{\mathbf{1366}}$};
\path  (S4-0) edge[black,thick] (S3-1);
\path  (S4-2) edge[black,thick] (S3-1);
\path  (S31-2) edge[black,thick] (S3-1);
\path  (S4-2) edge[black,thick] (S3-3);
\path  (S31-2) edge[black,thick] (S3-3);
\path  (S31-2) edge[black,thick] (S21p-3);
\path  (S31-2) edge[black,thick] (S21m-3);
\path  (S4-4) edge[black,thick] (S3-3);
\path  (S31-4) edge[black,thick] (S3-3);
\path  (S31-4) edge[black,thick] (S21m-3);
\path  (S31-4) edge[black,thick] (S21p-3);
\path  (S22p-4) edge[black,thick] (S21p-3);
\path  (S22m-4) edge[black,thick] (S21m-3);
\path  (S4-4) edge[black,thick] (S3-5);
\path  (S31-4) edge[black,thick] (S3-5);
\path  (S31-4) edge[black,thick] (S21p-5);
\path  (S31-4) edge[black,thick] (S21m-5);
\path  (S22p-4) edge[black,thick] (S21p-5);
\path  (S22m-4) edge[black,thick] (S21m-5);
\path  (S4-6) edge[black,thick] (S3-5);
\path  (S31-6) edge[black,thick] (S3-5);
\path  (S31-6) edge[black,thick] (S21p-5);
\path  (S31-6) edge[black,thick] (S21m-5);
\path  (S22p-6) edge[black,thick] (S21p-5);
\path  (S22m-6) edge[black,thick] (S21m-5);
\path  (S31-6) edge[black,thick] (S21p-7);
\path  (S31-6) edge[black,thick] (S21m-7);
\path  (S4-6) edge[black,thick] (S3-7);
\path  (S31-6) edge[black,thick] (S3-7);
\path  (S22p-6) edge[black,thick] (S21p-7);
\path  (S22m-6) edge[black,thick] (S21m-7);
\end{tikzpicture}
$$
\caption{Levels  \ $\ell = 0,\half,1,\ldots,3, \frac{7}{2}$ \  of the Bratteli diagram for the pair $(\A_4, \A_3)$.\label{fig:Abratteli}}
\end{figure} 

\begin{remark} Amazingly, the Bratteli diagrams in Figures \ref{fig:Sbratteli} and \ref{fig:Abratteli} also appear in the Schur-Weyl duality analysis of the McKay correspondence, as discussed in \cite{B} and \cite{BH3}.  The binary octahedral subgroup $\mathbf{O}$ of the special unitary group $\mathsf{SU}_2$  is the two-fold cover of the octahedral group, 
which is isomorphic to the symmetric group $\S_4$.   We use that fact to show that the Bratteli diagram for  tensor powers of the 2-dimensional spin module  of $\mathbf{O}$  (which is the defining module of $\mathbf{O}$ and $\mathsf{SU}_2$)
 is identical to Figure \ref{fig:Sbratteli}.   Similarly, the binary tetrahedral subgroup  $\mathbf{T}$ is the two-fold cover of tetrahedral group, which is isomorphic to the alternating group $\A_4$.  The Bratteli diagram for tensor powers of the 2-dimensional defining module
 of $\mathbf{T}$  is identical to Figure \ref{fig:Abratteli}. 
\end{remark}

\Yboxdim{5pt}  
\Ylinethick{.6pt} 

\begin{remark} The  \emph{tensor power Bratteli diagram} $\mathcal{B}_{\VV}(\GG)$ is constructed using the centralizer algebras $\Z_k(\GG) = \End_\GG(\VV^{\otimes k})$. The vertices on level $k$ of  $\mathcal{B}_{\VV}(\GG)$ are labeled by elements of $\Lambda_{k, \GG}$, and there are $c^\lambda_\mu$ edges from $\lambda \in \Lambda_{k,\GG}$ to $\mu \in \Lambda_{k+1,\GG}$ if $\GG^\lambda \otimes \VV \cong \bigoplus_{\mu \in \Lambda_\GG} c^\lambda_\mu \GG^\mu$.  In the special case that $\VV  = \Ind^\GG_\HH(\Res^\GG_\HH(\UU^0))$ and $\UU^0$ is the trivial $\GG$-module, $\mathcal{B}_\VV(\GG)$ is identical to  $\mathcal{B}(\GG,\HH)$ except that the half integer levels are missing from $\mathcal{B}_\VV(\GG)$. So for example,  in the tensor power Bratteli diagram that corresponds to Figure \ref{fig:Sbratteli},  there are two edges from the vertex $\yng(3,1)$ on level $1$ to the vertex $\yng(3,1)$ on level $2$. 
Including the intermediate half-integer levels,  which corresponds to performing restriction and then induction,  results in a 
diagram without multiple edges between vertices when $(\GG,\HH) = (\S_n,\S_{n-1})$ or $(\A_n,\A_{n-1})$,  since the restriction/induction rules for those pairs are multiplicity free.   The half-integer centralizer algebras  have proven to be a powerful tool in studying the structure of these 
tensor power centralizer algebras (for example, in \cite{HR} and \cite{BH3}), and we use them here to recursively derive dimension formulas. 
\end{remark}

\begin{section}{Dimensions formulas for symmetric group centralizer algebras}\end{section}

In the next two sections, we determine expressions for the dimensions of the irreducible modules 
for the
centralizer algebras in Table \ref{table1}.   Our arguments will invoke 
standard combinatorial facts  about representations of  the symmetric group $\S_n$.   The dimensions will be expressed as integer combinations of
Stirling numbers of the second kind.  We begin by briefly reviewing some known results about these numbers.  \medskip

\subsection{Stirling numbers of the second kind and Bell numbers}

There are several commonly used notations for Stirling numbers of the second kind;  for example, $S(k,t)$ is used by Stanley
 \cite{S1}.   In \cite{K}, Knuth remarks ``The lack of a widely accepted way to refer to these numbers has become almost scandalous,''  and he goes on to make a convincing argument for adopting the notation $\lb{k \atop t}\rb$, which we
will do here.  

The  \emph{Stirling number $\lb k \atop t \rb$   of the second  kind} counts the number  of ways to
partition a set of $k$ elements into $t$ disjoint nonempty blocks.    In particular, 
$\lb{k\atop 0}\rb = 0$ for all $k \geq 1$,  and $\lb{k \atop t} \rb = 0$ if $t > k$.    By convention,
$\lb{0 \atop 0}\rb = 1$.  These numbers satisfy the recurrence relations,  
\begin{eqnarray}
\lb{k+1 \atop t+1}\rb &=& \sum_{r=t}^k {k \choose r} \lb{r \atop t}\rb \label{eq:Stir1} \\
\lb{k+1 \atop t}\rb &=&t\lb{k \atop t}\rb + \lb{k \atop t-1}\rb.\label{eq:Stir2}\end{eqnarray}
For $k \geq 1$, 
\begin{equation} \label{eq:Bell1}
\sum_{t=0}^k \lb{k \atop t}\rb  =  \sum_{t=1}^k \lb{k \atop t}\rb  = \mathsf{B}(k),\end{equation} 
where $\mathsf{B}(k)$ is the $k$th {\it Bell number}.
More generally, for $k \geq 1$, 
\begin{equation}
 \label{eq:Bell2}
  \sum_{t=1}^n \lb{k \atop t}\rb \  =: \ \mathsf{B}(k,n) \end{equation} 
counts the number of ways to partition a set of $k$ elements into at most $n$ disjoint nonempty blocks, 
and $\mathsf{B}(k,n) = \mathsf{B}(k)$ if $n \ge k$.
Identifying  $\P_0(\para)$ with $\CC$, we have $\dimm \P_k(\para) = \mathsf{B}(2k)$  for all $k \in \ZZ_{\ge 0}$. 
In fact,   $\dimm \P_\ell(\para) = \mathsf{B}(2\ell)$ for all $\ell \in \frac{1}{2} \ZZ_{\ge 0}$, which can be seen
by taking $\nu = \emptyset$ in Corollary \ref{C:partdims} below. 
 
\begin{subsection}{Main result for symmetric group centralizer algebras}\end{subsection} 

Our aim in this section is to establish Theorem \ref{T:cent}, which gives the dimensions of
 the irreducible modules for the centralizer algebras $\Z_k(n)= \End_{\S_n}(\modu^{\ot k})$ and $\Z_{k+\half}(n)= \End_{\S_{n-1}}(\modu^{\ot k})$.  Throughout,  the notation  $\lam^\#= [\lam_2, \dots, \lam_{\ell(\lam)}]$ will designate a
 partition $\lam=[\lam_1,\lam_2, \dots, \lam_{\ell(\lam)}]$ with its largest part $\lam_1$ removed, and  $f^\lam$ will be 
 the number of standard tableaux of shape $\lam$, which is also the dimension of the irreducible  symmetric group module labeled by $\lam$.   If $\pi$
 is a partition contained in $\lambda$,  then $f^{\lambda/\pi}$ denotes the number of standard tableaux with skew shape 
 $\lambda \setminus \pi$. The Kostka number $\mathsf{K}_{\lambda,\gamma}$ counts the number of semistandard tableaux of shape $\lambda$ and type $\gamma$ (see Section \ref{sec:Sn}).
 
 \bigskip

\begin{thm} \label{T:cent} Let $k,n \in \ZZ_{\ge 0}$ and $n \geq 1$,  and  let  the notation be as in Table \ref{table1}.

  \begin{itemize}
\item[{\rm (a)}]   Assume   $\lambda = [\lam_1, \dots, \lam_n] \vdash n$, and  $\lam \in \Lambda_{k,\S_n}$.     Then 
\begin{align*}  \dimm \Z_{k,n}^\lam  \  &= \  \sum_{t=0}^{n} \lb{k \atop t}\rb \, \KK_{\lam, [n-t, 1^t]} \ = \ \sum_{t=|\lam^\#|}^n\lb{k \atop t}\rb \, f^{\lam/[n-t]} \\
&= \  f^{\lam^\#} \, \sum_{t = {|\lam^\#|}}^{n-\lam_2} \lb{k \atop t}\rb \, {t \choose {|\lam^\#|}} 
 \ + \ \sum_{t=n-\lam_2+1}^n   \lb{k \atop t}\rb \, f^{\lam/[n-t]}.
 \end{align*}
\item[{\rm (b)}]  Assume  $\mu = [\mu_1, \ldots, \mu_{n-1}] \vdash n-1$, and  $\mu \in \Lambda_{k+\half, \S_{n-1}}$.    Then
\begin{align*} \dimm \Z_{k+\half,n}^\mu \ &= \  \sum_{t=0}^{n-1} \lb{k+1 \atop t+1}\rb \, \KK_{\mu, [n-1-t,1^t]} 
= \sum_{t=|\mu^\#|}^{n-1}\lb{k+1 \atop t+1}\rb \, f^{\mu/[n-1-t]} \\
&= \  f^{\mu^\#} \, \sum_{t = {|\mu^\#|}}^{n-1-\mu_2} \lb{k+1 \atop t+t}\rb \, {t \choose {|\mu^\#|}} 
 \ + \ \sum_{t=n-\mu_2}^n   \lb{k+1 \atop t+1}\rb \, f^{\mu/[n-1-t]}.
 \end{align*}
 
\item[{\rm (c)}]  $\dimm \Z_k(n) \ =\  \dimm \Z_{2k,n}^{[n]} = \displaystyle{ \sum_{t =0}^{n} \lb{2k \atop t}\rb}  \ = \mathsf{B}(2k,n)  \qquad  (\ =\ \mathsf{B}(2k) \ \  \hbox{\rm if} \  \ n \geq 2k\ ).$ 
\item[{\rm (d)}] $ \dimm \Z_{k+\half}(n) \ = \ \dimm \Z_{2k+\half,n}^{[n-1]} 
 = \displaystyle{ \sum_{t =0}^{n-3} \lb{2k +1\atop t+1}\rb + \left(\lb{2k+1 \atop  n-1}\rb + \lb{2k+1 \atop  n}\rb\right)}$\ 
 
$\qquad \qquad \qquad =  \ \displaystyle{\sum_{t =1}^{n} \lb{2k +1\atop t}\rb \ = \mathsf{B}(2k+1,n) 
  \qquad  (\ =\ \mathsf{B}(2k+1) \ \  \hbox{\rm if} \  \ n \geq 2k+1\ ). }$ 
 \end{itemize}
\end{thm} 
 
\begin{remark}  When $n > k$, the top limit in the summation in part (a) can be taken to be $k$ 
as the Stirling numbers $\lb{k \atop t}\rb$
are 0 for $t > k$.   When $n \leq k$, the
term $[n-t,1^t]$ for $t = n$ should be interpreted as the partition $[1^n]$.  
  In that special case,  $\KK_{\lam, [1^n]} = f^\lam$, the
number standard tableaux of shape $\lam$, as each entry in the tableau appears only once.   The term $t = n-1$ gives the same
Kostka number  $\KK_{\lam, [1^n]} = f^\lam$. 
The only time that the term $t=0$ contributes is when $k = 0$.  The Stirling number $\lb{0 \atop 0}\rb = 1$, and the Kostka number 
$\KK_{\lam, [n]} = 0$ if $\lam \neq [n]$ and $\KK_{[n],[n]} = 1$.     Thus,  $\dimm \Z_0^\lam(n) = \delta_{\lam,[n]}$,
as expected, since $\modu^{\ot 0} = \S_n^{[n]}$ by definition.  In the proof to follow, we will assume $k \geq 1$. 
\end{remark}  
\begin{proof}  (a)  For 
$1 \leq t \leq n$,  the linear span of $t$-element ordered subsets of $\{1,2,\dots, n\}$ forms an $\S_n$-module  
isomorphic  to the permutation module $\perm^{[n-t,1^t]}$.  For  $t = n-1$ and $t = n$,  both modules are isomorphic to $\perm^{[1^n]}$.   We claim that 
\begin{equation}\label{eq:permmods} \modu^{\ot k}  = \sum_{t=1}^n \lb{k \atop t}\rb \perm^{[n-t,1^t]}. \end{equation}
This can be seen as follows:   Let $\mathsf{u}_1,\dots \mathsf{u}_n$ be the basis for $\modu$ that $\S_n$ permutes.   For each set partition
of $\{1,\dots, k\}$ into $t$ blocks, we get a copy of $\perm^{[n-t,1^t]}$ spanned by the vectors 
$\mathsf{u}_{j_1} \ot \mathsf{u}_{j_2} \ot \cdots \ot \mathsf{u}_{j_k}$, where $j_a = j_b$ if and only if $a,b$ are in the same part of the set partition.   
There are $\lb{k \atop t}\rb$ such set partitions.  
The multiplicity of $\S_n^\lam$ in $\modu^{\ot k}$ is obtained from \eqref{eq:permmods} by observing that 
$\S_n^\lam$ has multiplicity $\KK_{\lam,[n-t,1^t]}$ in $\perm^{[n-t,1^t]}$.   By Schur-Weyl duality, the multiplicity of
$\S_n^\lam$ in $\modu^{\ot k}$ equals $\dimm \Z_k^\lam(n)$, and therefore, 
$\dimm \Z_{k,n}^\lam  = \sum_{t=1}^{n} \lb{k \atop t}\rb \KK_{\lam, [n-t, 1^t]}.$ 

The second equality in part (a)  follows from the fact that $\KK_{\lam, [n-t, 1^t]} = 0$
unless  $\lam_1 \geq n-t$, i.e.  unless $t \geq n-\lam_1 = |\lam^\#|$, and from the fact that a semistandard
tableau, whose entries are  $n-t$ zeros and the numbers $1, 2, \dots, t$,  must have the $n-t$ zeros in the first row
and have a standard filling of the skew shape $\lam/[n-t]$.    To see that the last line of part (a) holds, observe that 
when $n-t \geq \lam_2$, any standard tableau of shape $\lam/[n-t]$, has $\lam_1 - (n-t) = t-(n-\lam_1)$ 
entries chosen from $\{1,2, \dots, t\}$ in its first row.    There are 
$${t \choose t-(n-\lam_1)}  = {t \choose n-\lam_1}  = {t \choose |\lam^\#|}$$ ways to select those entries.    
The remaining integers from $\{1,2,\dots, t\}$  fill the shape $\lam^\#$ to give a standard tableau.   Therefore,
$f^{\lam/[n-t]} =  {t \choose |\lam^\#|} f^{\lam^\#}$ if  $n-\lam_2 \geq t$.  
  
For part (b),   identifying $\S_{n-1}$ with the permutations of $\S_n$ that fix $n$, we see that restriction from $\S_n$ to $\S_{n-1}$
 gives  $\modu = \mathsf{M}_{n-1} \oplus \CC \mathsf{u}_n$, where $\mathsf{M}_{n-1}$ is the permutation module of $\S_{n-1}$ spanned
 by the vectors $\mathsf{u}_1, \dots, \mathsf{u}_{n-1}$.   Hence,
 $\modu^{\ot k} \cong \bigoplus_{s=0}^k  {k \choose s} \mathsf{M}_{n-1}^{\ot s}$ as an $\S_{n-1}$-module,  which together with (a)  implies 
 
\begin{align*} 
\dimm \Z_{k+\half,n}^\mu  &=   \sum_{s=0}^{k}  {k \choose s}\,   \dimm \Z_{s,n-1}^\mu \\ 
& =\sum_{s=0}^{k}  {k \choose s}\left(\sum_{t=1}^{n-1}  \lb{s \atop t}\rb \mathsf{K}_{\mu,[n-1-t,1^t]}\right) \\ 
& = \sum_{t=0}^{n-1} \left( \sum_{s=0}^k  {k \choose s}\ \lb{s \atop t}\rb\right) \mathsf{K}_{\mu,[n-1-t,1^t]}   \\
 & = \sum_{t=0}^{n-1} \left( \sum_{s=t}^k  {k \choose s}\ \lb{s \atop t}\rb\right) \mathsf{K}_{\mu,[n-1-t,1^t]}  \\
 & =  \sum_{t=0}^{n-1} \lb{k+1\atop t+1}\rb  \mathsf{K}_{\mu,[n-1-t,1^t]}   \qquad \hbox{\rm using 
\eqref{eq:Stir1}}. 
 \end{align*}  
This establishes the first  equality  in (b)  for $k$ and all $n\geq 1$.  The remainder of (b) can 
be shown by arguments similar to the ones used for part (a). 
  
Part (c) is an immediate consequence of (a),  since $\S_n^{[n]}$ is the trivial $\S_n$-module, 
and $\modu^{\ot k}$ is isomorphic, as an $\S_n$-module,  to its dual module, so that
\begin{align*} \dimm \Z_k(n) &= \dimm \Z_{2k,n}^{[n]} = \sum_{t=0}^{n} \lb{2k \atop t}\rb \KK_{[n], [n-t, 1^t]} \\
 &= \sum_{t=0}^n \lb{2k \atop t}\rb = \mathsf{B}(2k,n) \qquad 
 \left (= \mathsf{B}(2k) \ \  \hbox{\rm if} \  \ n \geq 2k\ \right ). \end{align*}
Part (d) follows readily from (b)
for similar reasons.   \end{proof}  

\begin{remark}  In  \cite[Prop. 2.1]{D},  it was shown using the exponential generating functions of
Goupil and Chauve \cite{GC}  that for the partition $\lam = [\lam_1, \lam_2, \dots, \lam_{\ell(\lam)}]$,  the multiplicity of $\S_n^\lam$ in $\modu^{\ot k}$ (i.e. $\dimm \Z_{k,n}^\lam$) equals  $
f^{\lambda^\#} \,\sum_{t =\vert \lambda^\#\vert }^{n-2} {t \choose {\vert \lambda^\#\vert}} \lb{k \atop t}\rb$
whenever $1 \leq k \leq n-\lambda_2$.
This is a special case of part (a) of Proposition \ref{P:nuprop}.   As mentioned earlier, this result was used 
in \cite{D} to bound the mixing time of a Markov chain on $\S_n$.  \end{remark}

 \begin{remark} Suppose $\nu = [\nu_1,\dots, \nu_{\ell(\nu)}]$ \,  is a partition with \, $0 \leq |\nu| \le k$, and 
for \, $n \geq 2k$,  let  $[n-|\nu|,\nu]$ be the partition of $n$ given by \,  $[n-|\nu|, \nu] : = [n-|\nu|, \nu_1, \ldots, \nu_{\ell(\nu)}]$.  
In the next proposition, we obtain an expression for the dimension of the irreducible $\Z_k(n)$-module 
$\Z_{k,n}^{[n-|\nu|,\nu]}$ (and for the $\Z_{k+\half}(n)$-module
$\Z_{k+\half,n}^{[n-1-|\nu|,\nu]}$ when $n-1 \ge 2k$).   We 
prove  that both dimensions equal $f^\nu =\dimm \S_k^\nu$ when $|\nu| = k$.      
When  $\nu = [k]$, $\dimm \Z_{k,n}^{[n-k,k]} = f^{[k]} =1$ for all $n \ge 2k$.     When $n=2k-1$,  the kernel of the
map $\P_k(n)  \to \Z_k(n)$ is one-dimensional,  since $[n-k,k]$ is not a partition in that case.     
In \cite{BH2},  we describe the kernel of the map $\P_k(n)  \to \Z_k(n)$ for all $n < 2k$.   \end{remark}
\newpage

\begin{prop}\label{P:nuprop} Assume $\nu =[\nu_1, \dots, \nu_{\ell(\nu)}]$ is a partition with $0 \leq |\nu| \le k$.  
\begin{itemize}
\item[{\rm (a)}]   If $0 \leq 2k \le  n$,   then 
\begin{equation}\label{eq:propa}\dimm \Z_{k,n}^{[n-|\nu|,\nu]} = f^{\nu}\,\sum_{t =\vert \nu\vert }^{k} {t \choose {\vert \nu\vert}} \lb{k \atop t}\rb \quad \big(= f^\nu \ \ \hbox{when} \ \
 |\nu| = k\big).\end{equation} 
\item[\rm{(b)}]   If $0 \le 2k \le n-1$,  then
\begin{equation}\label{eq:propb}\dimm \Z_{k+\half,n}^{[n-1-|\nu|,\nu]} =  f^{\nu}\, \sum_{t = |\nu|}^{k} {t \choose |\nu|} \, \lb{k+1\atop t+1}\rb  \quad \big(= f^\nu \ \ \hbox{when} \ \
 |\nu| = k\big).\end{equation}
\end{itemize}
\end{prop} 
\begin{proof}
 (a)   From Theorem \ref{T:cent}\,(a),  we know for the partition $ [n-|\nu|,\nu] = [n-|\nu|, \nu_1, \dots, \nu_{\ell(\nu)}]$ that 
\begin{equation}\label{eq:nueq} \dimm \Z_{k,n}^{[n-|\nu|,\nu]} = f^{\nu} \, \sum_{t = {|\nu|}}^{n-\nu_1} \lb{k \atop t}\rb \, {t \choose {|\nu|}} 
 \ + \ \sum_{t=n-\nu_1+1}^n   \lb{k \atop t}\rb \, f^{[n-|\nu|,\nu]/[n-t]}.
\end{equation}
Since  $\nu_1 \leq |\nu| \leq k$, and we are assuming $n \geq 2k$,   it follows that  $n-\nu_1 \geq n-k \geq k$.   Thus, the
first summation equals  $f^{\nu} \, \sum_{t = {|\nu|}}^{k} \lb{k \atop t}\rb \, {t \choose {|\nu|}}$, and the second is 0, which
is the assertion in (a).     The argument for part (b) is completely analogous.  
\end{proof}   
  
The partition algebras $\P_k(\para)$ and $\P_{k+\half}(\para)$ are generically semisimple for 
all $\para\in\CC$ with  $\para\not \in \{0,1,\dots, 2k-1\}$ (see \cite{MS} or \cite[Thm.~3.7]{HR}.   Assume  $\nu = [\nu_1, \dots, \nu_{\ell(\nu)}]$ is a partition with $0 \leq \vert \nu \vert \leq k$, and let  $\P_{k,\para}^\nu$ denote the irreducible $\P_k(\para)$-module
and $\P_{k+\half,\para}^\nu$ denote the irreducible $\P_{k+\half}(\para)$-module  indexed by $\nu$.
The dimension of $\P_{k,\para}^\nu$ (resp. $\P_{k+\half,\para}^\nu$) is the same for all generic values of $\para$. 
Therefore, we can apply Proposition \ref{P:nuprop} with $n = 2k$ and $\lambda = [n-\vert \nu \vert, \nu_1, \ldots, \nu_{\ell(\nu)}] \vdash n$  to conclude the following:   

\begin{cor} \label{C:partdims}   Let  $\nu$ be a partition with $0 \leq \vert \nu \vert \leq k$.   For $\para \not \in 
\{0,1,\ldots, 2k-1\}$,  let  $\P_{k,\para}^\nu$ denote the irreducible $\P_k(\para)$-module
and $\P_{k+\half,\para}^\nu$ denote the irreducible $\P_{k+\half}(\para)$-module  indexed by $\nu$.    Then
\vspace{-.3cm}  \begin{align*}
\dimm \P_{k,\para}^\nu &= f^{\nu}\,\sum_{t =\vert \nu\vert }^{k} {t \choose {\vert \nu\vert}} \lb{k \atop t}\rb\qquad \qquad  \big(= f^\nu \ \ \hbox{when} \ \
 |\nu| = k\big)  \\
\dimm \P_{k+\half,\para}^\nu &= f^{\nu} \, \sum_{t =  {\vert \nu\vert}}^{k} {t \choose {\vert \nu\vert}} \, \lb{k+1\atop t+1}\rb \qquad  \big(= f^\nu \ \ \hbox{when} \ \
 |\nu| = k\big).
\end{align*} \end{cor}

\subsection{Bijective proof of Theorem \ref{T:cent} (a)}\label{ss:bijection}

By Theorem \ref{T:cent}(a),  we know that $\dimm \Z_{k,n}^\lambda$ equals the number of pairs $(P,T)$ where $P$ is a set partition of $\{1, 2, \ldots, k\}$ into $t$ blocks for some $t \in \{1,\dots,n\}$,  and $T$ is a semistandard tableau of shape $\lambda$  filled with $n-t$ zeros and 
 $t$ distinct numbers from $\ZZ_{>0}$, so that $T$ has type $[n-t,1^t]$.  By Section \ref{sec:Bratteli},  we know that this dimension is also equal to the number of paths from the root of the Bratteli diagram $\mathcal{B}(\S_n,\S_{n-1})$ to $\lambda \in \Lambda_{k,\S_n}$. Such paths (also referred to as vacillating tableaux) are given by a sequence of partitions
$$
\left(\lambda^{(0)} = [n],\, \lambda^{(\half)} = [n-1],\, \lambda^{(1)},\lambda^{(1+\half)}, \ldots, \lambda^{(k-1)}, \lambda^{(k - \half)}, \lambda^{(k)} = \lambda\right)
$$
such that $\lambda^{(i)} \in \Lambda_{i,\S_n}$, \, $\lambda^{(i-\half)} \in \Lambda_{i-\half,\S_{n-1}}$ for each $i$, and
\begin{enumerate}
\item[(a)] $\lambda^{(i-\half)} = \lambda^{(i-1)} - \square$,
\item[(b)] $\lambda^{(i)} = \lambda^{(i - \half)} + \square$,
\end{enumerate}
for each integer $1 \le i \le k$. In this section, we demonstrate a bijection between  paths and pairs $(P,T)$,  thereby giving a combinatorial proof of Theorem \ref{T:cent}(a). The corresponding bijection for Theorem \ref{T:cent}(b) is gotten by applying this same bijection on paths to $\mu \in \Lambda_{k-\half,\S_{n-1}}$. We assume familiarity with the RSK row-insertion algorithm (see for example \cite[Sec.~7.11]{S2}), and let  $T \RSK b$ denote row insertion of the integer $b$ into the semistandard tableau $T$.  The bijection here, which works for all $n \ge 1$ and $k \ge 0$,  extends that of \cite[Thm.~2.4]{CDDSY}, which holds for $n \ge 2k$.   It  is easily adaptable to give a combinatorial proof
of Theorem \ref{T:cent}\,(b). 

\bigskip
\noindent{\bf Bijection from paths to pairs $\boldsymbol{(P,T)}$}: \ Given a path $(\lambda^{(0)}, \lambda^{(\half)},\ldots,\lambda^{(k)}=\lambda)$ to $\lambda \in \Lambda_{k,\S_n}$, we recursively construct a sequence $(P_0, T_0), (P_\half,T_\half), (P_1, T_1), \ldots, (P_k, T_k)$ such that, for each $i$, $P_i$ is a set partition of $\{1, \ldots, \lceil i \rceil\}$ into $t$ blocks,  and $T_i$ is a semistandard tableau of shape $\lambda^{(i)}$ with $n-t$ zeros and nonzero entries from the set $\mathsf{max}(P_i)$  whose elements are the maximal entries in the $t$ blocks of $P_i$.  Then  $(P_k,T_k)$ is the pair associated with the path $(\lambda^{(0)}, \lambda^{(\half)},\ldots,\lambda^{(k)}=\lambda)$.

Let $P_0 = \emptyset$ and let $T_0$ be the semistandard  tableau of shape $[n]$ and type $[n]$, i.e., with each entry equal to 0.
Then for each integer $i = 1, 2, \ldots, k$,  perform these steps.

\begin{itemize}
\item[(1)] Construct $(P_{i - \half}, T_{i-\half})$ from  $(P_{i - 1}, T_{i-1})$ as follows:  \ Let $b$ be the unique nonegative integer such that $T_{i-1} = (T_{i - \half} \RSK b)$. Since $b \in T_{i-1}$,  we know that $0 \le b < i$. If $b = 0$, then $P_{i - \half}$ is obtained by adding the block $\{i\}$ to $P_{i-1}$. If $b > 0$, then $P_{i - \half}$ is obtained by adding $i$ to the block that contains $b$ in $P_{i-1}$.

\item[(2)] Construct $(P_i,T_i)$ from $(P_{i - \half}, T_{i-\half})$ by letting  $P_i$ equal  $P_{i-\half}$ and $T_i$ be the column strict tableau obtained  from $T_{i-\half}$ by adding the entry $i$ in the box  $\lambda^i \setminus \lambda^{i - \half}$.
\end{itemize}
By the above construction,  $P_i$ is a set partition of $\{1, \ldots, \lceil i \rceil\}$ for each $i$,  and if $P_i$ has $t$ parts,  then $T_i$ is a semistandard tableau with $n-t$ zeros and with the elements of $\mathsf{max}(P_i)$ as its nonzero entries. 

The map is bijective since the above construction can be reversed:    Given a pair $(P,T)$ consisting of a set partition $P$ of $
\{1,\dots,k\}$ into $t$ blocks and a semistandard tableau $T$ of shape $\lambda \in \Lambda_{k,\S_n}$ filled with $n-t$ zeros and the elements of $\mathsf{max}(P)$, we produce a path $(\lambda^0=[n], \lambda^{\half},\ldots,\lambda^k=\lambda)$  by performing these steps:
Start with $P_k = P,$ $T_k = T$,  and work backwards to produce the sequence $(P_k,T_k),(P_{k-\half},T_{k - \half}), $ $(P_{k-1},T_{k-1}), \ldots, (P_0,T_0)$ as follows:
\begin{enumerate}
\item[$(1)'$] Construct $(P_{i-\half},T_{i - \half})$ from $(P_{i},T_{i })$ by letting $P_{i-\half} = P_i$ and deleting $i$ from $T_i$.
\item[$(2)'$] Construct $(P_{i-1},T_{i - 1})$ from $(P_{i-\half},T_{i -\half})$  by the following procedure: \  If $i$ is a singleton block in $P_{i-\half}$,  then $T_{i - 1} = (T_{i - \half} \RSK 0)$. If $i$ is not a singleton block in $P_{i - \half}$,  then  $T_{i - 1} = (T_{i - \half} \RSK b)$, where $b$ is the second largest element of the block containing $i$. Let $P_{i-1}$ be obtained by deleting $\{i \}$ from $P_{i-\half}$.
\end{enumerate}
If  $\lambda^{(i)}$ is the partition shape of $T_i$ for $i=0,\half, 1, \dots, k$,   then $(\lambda^{(0)}, \lambda^{(\half)}, \ldots, \lambda^{(k)})$ to $\lambda$ is the corresponding path in the Bratteli diagram.  This bijection is illustrated in the next example.

\Yboxdim{9.5pt} 
\begin{example}\label{ex:bijection} If $n= 4, k = 3,$ and  $\lambda = [2,2]$, then
\begin{align*}
\dimm \Z_{3,4}^{[2,2]}
&   =  \lb 3 \atop 1 \rb \mathsf{K}_{[2,2],[3,1]} + \lb 3 \atop 2 \rb \mathsf{K}_{[2,2],[2,1,1]} +   \lb 3 \atop 3 \rb \mathsf{K}_{[2,2],[1,1,1,1]} 
 =  1  \cdot 0 + 3 \cdot 1 + 1 \cdot 2   = 5.
\end{align*}
This is the subscript $5$ on $[2,2]$ at level 3 in the Bratteli diagram in Figure \ref{fig:Sbratteli}.
The five corresponding pairs $(P,T)$ of set partitions and semistandard tableaux are:
$$
\big(\{1,2\mid3\},  \young(00,23)\big), \quad
\big(\{1,3\mid2\} ,  \young(00,23)\big), \quad
\big( \{1\mid2,3\} ,  \young(00,13)\big), \quad
\big( \{1\mid2\mid3\},  \young(02,13)\big), \quad
\big(\{1\mid2\mid3\},   \young(01,23)\big).
$$
The five paths to $[2,2] \in \Lambda_{3,\S_4}$ and the corresponding bijections with these pairs  are illustrated in Figure \ref{fig:bijection}.

\begin{figure}[h!]
\begin{equation*}
\begin{array}{l|ll|ll|ll|ll|l|}
\text{level} & p_1 && p_2  && p_3 && p_4 && p_5 \\
\hline
\ell=0\quad &  \young(0000)    && \young(0000)    && \young(0000)    && \young(0000) && \young(0000) \\
              &  \emptyset              && \emptyset               && \emptyset                && \emptyset        &&  \emptyset   \\ \hdashline
\ell=\half &      \young(000)\RSK0    && \young(000)\RSK0    && \young(000)\RSK0    && \young(000)\RSK0 &   & \young(000)\RSK0 \\ 
              &  \{1\}               &&  \{1\}                &&  \{1\}                 &&  \{1\}         &&   \{1\}     \\ \hdashline
\ell=1 &           \young(0001)    && \young(000,1) && \young(000,1) && \young(000,1) && \young(000,1) \\ 
              &  \{1\}               &&  \{1\}                &&  \{1\}                 &&  \{1\}         &&   \{1\}     \\ \hdashline
\ell=1\half &    \young(000)\RSK1   && \young(001)\RSK0    && \young(00,1)\RSK0   && \young(00,1)\RSK0 &&  \young(00,1)\RSK0\\ 
             &  \{1,2\}               &&  \{1\mid2\}                &&  \{1\mid2\}                 &&  \{1\mid2\}         &&   \{1\mid2\}     \\ \hdashline
\ell=2&           \young(000,2) && \young(001,2) && \young(002,1) && \young(00,12) && \young(00,1,2) \\  
             &  \{1,2\}               &&  \{1\mid2\}                &&  \{1\mid2\}                 &&  \{1\mid2\}         &&   \{1\mid2\}     \\ \hdashline
\ell=2\half &    \young(00,2)\RSK0   && \young(00,2)\RSK1&& \young(00,1)\RSK2 && \young(02,1)\RSK0 &&  \young(01,2)\RSK0 \\  
             &  \{1,2\mid3\}               &&  \{1,3\mid2\}                &&  \{1\mid2,3\}                 &&  \{1\mid2\mid3\}         &&   \{1\mid2\mid3\}     \\ \hdashline
\ell=3 &           \young(00,23) && \young(00,23) && \young(00,13) && \young(02,13) && \young(01,23) \\ 
             &  \{1,2\mid3\}               &&  \{1,3\mid2\}                &&  \{1\mid2,3\}                 &&  \{1\mid2\mid3\}         &&   \{1\mid2\mid3\}     \\   \hline
\end{array}
\end{equation*}
\caption{The bijection between the five paths to $\lambda = [2,2] \in \Lambda_{3,\S_4}$ in the Bratteli diagram $\mathcal{B}(\S_4,\S_3)$ and pairs $(P,T)$ of Example \ref{ex:bijection} consisting of a set partition $P$ of $\{1,2,3\}$ into $t$ blocks and a semistandard tableau $T$ filled with $4-t$ zeroes and the maximum entries of the blocks of $P$. \label{fig:bijection}}
\end{figure}
\end{example}

\subsection{Dimension Examples} 

\begin{example} Assume $\lam = [\lam_1,\lam_2, \ldots, \lam_{\ell(\lam)}] \vdash n$ and $1 \le\lam_2 \le 2$.   We claim that under these assumptions, Theorem \ref{T:cent}(a) simplifies to \begin{align}\begin{split}\label{eq:lam2small}
 \dimm \Z^{\lambda}_{k,n}  
& = \begin{cases}  \displaystyle{ f^{\lambda^\#} \sum_{t = |\lambda^\#|}^{n-2} \binom{t}{|\lambda^\#|} \lb{k \atop t}\rb + f^\lambda 
\bigg(\lb k \atop n-1\rb + \lb k \atop n\rb\bigg)}  & \qquad  \hbox{if} \ \  \lam_1 > 1,  \\
\displaystyle{\lb k \atop n-1\rb + \lb k \atop n\rb}  & \qquad  \hbox{if} \ \  \lam_1 = 1, \end{cases}\end{split}
\end{align}
where $\lambda^\# = [\lambda_2, \ldots, \lambda_{\ell(\lambda)}]$. In particular, this formula holds for all $\lambda$ when $n \le 5$.
When $\lam_1 = 1$, then  $\lam = [1^n]$, and this says  $\dimm \Z^{[1^n]}_{k,n} = \displaystyle{\lb k \atop n-1\rb + \lb k \atop n\rb}$ for all $n \ge 1$.

To verify this assertion,  we start from the last line of Theorem \ref{T:cent} (a),
\begin{equation}\label{eq:lastline}\dimm \Z_{k,n}^\lam = f^{\lam^\#} \, \sum_{t = {|\lam^\#|}}^{n-\lam_2} \lb{k \atop t}\rb \, {t \choose {|\lam^\#|}} 
 \ + \ \sum_{t=n-\lam_2+1}^n   \lb{k \atop t}\rb \, f^{\lam/[n-t]}.\end{equation}  When $\lam_1 = 1$, then $\lam = [1^n]$ and $\lam^\# = [1^{n-1}]$,  
and  this reduces to 
$$\dimm \Z_{k,n}^{[1^n]} = f^{[1^{n-1}]} \lb{k \atop n-1}\rb \ + \ f^{[1^n]}\lb{k \atop n}\rb = \lb{k \atop n-1}\rb \ + \ \lb{k \atop n}\rb.$$ 
When $\lam_1 > 1$, and $\lam_2 = 1$,   then  $\lam^\# = [1^{n-\lam_1}]$,  and 
\begin{align*} \dimm \Z_{k,n}^\lam &= f^{\lam^\#} \, \sum_{t = {|\lam^\#|}}^{n-1} \lb{k \atop t}\rb \, {t \choose {|\lam^\#|}} 
 \ + \  f^{\lam}  \lb{k \atop n}\rb \\
&= f^{\lam^\#} \, \sum_{t = {|\lam^\#|}}^{n-2} \lb{k \atop t}\rb \, {t \choose {|\lam^\#|}} \ + \  f^{\lam^\#}  \lb{k \atop n-1}\rb {n-1 \choose 
{|\lam^\#|}} 
 \ + \  f^{\lam}  \lb{k \atop n}\rb \\ 
 & = f^{\lam^\#} \sum_{t = |\lam^\#|}^{n-2} \binom{t}{|\lambda^\#|} \lb{k \atop t}\rb + f^\lambda 
\bigg(\lb{k \atop n-1}\rb + \lb{k \atop n}\rb\bigg), 
 \end{align*} 
since $f^{\lam} = f^{\lam^\#}{n-1 \choose n-\lam_1}$.  Finally, when $\lam_2 = 2$,  the assertion is exactly \eqref{eq:lastline},
as $f^{\lam/[1]} = f^\lam$.   
 \end{example}

\begin{example} Since $f^{[n]} = 1$,  {\rm Corollary  \ref{C:partdims}  implies for all $k \geq n$ and all
generic values of  $\para$  that \begin{equation}\label{eq:generic} \dimm \P_{k,\para}^{[n]} = \sum_{t=n}^k {t \choose n} \lb{k \atop t}\rb  \quad  \hbox{and} \quad  
\dimm\P_{k + \half,\para}^{[n]}  =   \sum_{t=n}^{k}{t \choose n} \lb{k+1 \atop t+1}\rb.\end{equation}  
The last line  \eqref{eq:lastline} of part (a) of Theorem \ref{T:cent}  gives
\begin{equation}\label{eq:trivial} 
 \dimm \Z_{k,n}^{[n]} =  \sum_{t=0}^n \lb{k \atop t}\rb = \mathsf{B}(k,n)\qquad  \left( \, =  \mathsf{B}(k) \ \, \hbox {when}  \ \,  n \geq k\right),  
 \end{equation} 
 while the last line of part (b) of Theorem \ref{T:cent} says
\begin{equation}\label{eq:halftriv}
\hspace{-.5cm}  \dimm \Z_{k+\half,n}^{[n-1]} = \sum_{t=1}^n \lb{k+1 \atop j}\rb \  = \mathsf{B}(k+1,n) \qquad  \left( =  \mathsf{B}(k+1)\ \hbox {when}  \  n \geq k+1\right).   
\end{equation}} 
 \end{example}

\begin{remark}{\rm  A \emph{(Gelfand) model}  for an algebra is a  module in which each irreducible module
appears as a direct summand with multiplicity one.   In \cite{HRe},  Halverson and Reeks construct models for certain diagram algebras,
including  the partition algebras $\P_k(\para)$ for generic $\para$, 
using basis diagrams invariant under reflection about the horizontal
axis (the symmetric diagrams) and  the diagram conjugation action of $\P_k(\para)$ on them.   
The model $\Ms_{\P_k}$ for $\P_k(\para)$  decomposes into submodules $\Ms_{\P_k} =\bigoplus_{r,p} \Ms^{r,p}_{\P_k}$,
where $\Ms^{r,p}_{\P_k} = \bigoplus_{{\nu \vdash r} \atop {\mathsf{odd}(\nu) = p}}  \P_{k,\para}^\nu$
according to the size $r = |\nu|$ of the partition $\nu$ and its number  $p = \mathsf{odd}(\nu)$ of odd parts.   By enumerating symmetric diagrams,  
they determine that 
\begin{equation}\label{eq:model} \dimm \Ms^{r,p}_{\P_k} = \sum_{t= r}^k  {r \choose p}\,(r-p-1)!!\, {t \choose r}\lb{k \atop t}\rb =  {r \choose p}\,(r-p-1)!!\,\sum_{t = r}^k{t \choose r}\lb{k \atop t}\rb, \end{equation}
where $r-p$ is even and $(r-p-1)!! = (r-p-1)(r-p-3) \cdots 3 \cdot 1$.  The factor ${r \choose p}\,(r-p-1)!!$ comes from 
 the fact (see \cite{HRe}) that 
\begin{equation}\label{eq:invol} \sum_{{\nu \vdash r} \atop {\mathsf{odd}(\nu) = p}}  f^\nu = | \mathrm{I}^{r,p} \vert = {r \choose p}\,(r-p-1)!!, 
\end{equation}
where $ \mathrm{I}^{r,p}$ is the set of involutions (elements of order 2) with $p$ fixed points in the symmetric group $\S_r$.
Corollary \ref{C:partdims}  and \eqref{eq:invol}  give an alternate proof of \eqref{eq:model}: 
\begin{equation*} 
\dimm \Ms^{r,p}_{\P_k} =\!\!\!\! \sum_{{\nu \vdash r} \atop {\mathsf{odd}(\nu) = p}}\!\!\! \dimm \P_{k,\para}^\nu 
= \!\!\!\! \sum_{{\nu \vdash r} \atop {\mathsf{odd}(\nu) = p}}\!\!\!  f^\nu \, \Bigg( \sum_{t=\vert \nu \vert}^k 
 {t \choose \vert \nu \vert}\lb{k \atop t}\rb\Bigg) 
 ={r \choose p}(r-p-1)!! \,\!\!  \sum_{t=r}^k 
 {t \choose r}\lb{k \atop t}\rb .  \end{equation*}
 }
\end{remark}   
 
  \bigskip  
\begin{section}{Dimension formulas for alternating group centralizer algebras}\end{section}  

The restriction rules in  \eqref{eq:branch} combined with Theorem \ref{T:res-ind}\,(b)  can be used to derive expressions for the  dimensions of the irreducible modules for the alternating group centralizer algebras $\WZ_k(n)$ and  $\WZ_{k+\half}(n)$ from the dimension formulas for irreducible modules for  $\Z_k(n)$ and  $\Z_{k+\half}(n)$ in Theorem \ref{T:cent}.

\begin{thm}\label{T:Ardims} Assume $k \in \mathbb{Z}_{\geq 0}$.  The dimensions of the irreducible modules for $\WZ_k(n)$ and  $\WZ_{k+\half}(n)$ are as follows (using notation from Table \ref{table1}).
\begin{itemize}  
\item[{\rm (a)}] For $\lam\vdash n$ and $\lam \in \Lambda_{k,\A_n}$,
$$
\begin{array}{ll}
\dimm \WZ_{k,n}^\lam  = \dimm \Z_{k,n}^\lam  +  \dimm \Z_{k,n}^{\lam^{\ast}}, \quad & \text{if $\lam > \lam^{\ast}$,} \\
\dimm \WZ_{k,n}^{\lam^{+}} = \dimm \WZ_{k,n}^{\lam^{-}} = \dimm \Z_{k,n}^{\lam}, \quad & \text{if $\lam = \lam^{\ast}$,} \\
\end{array}
$$
where $\dimm \Z_{k,n}^\lam$ and $\dimm \Z_{k,n}^{\lam^{\ast}}$ are given by the formula in Theorem \ref{T:cent}\,(a).

\item[{\rm (b)}] For $\mu \vdash n-1$ and $\mu \in \Lambda_{k,\A_{n-1}}$,
$$
\begin{array}{ll}
\dimm \WZ_{k+\half,n}^\mu  = \dimm \Z_{k+\half,n}^\mu  +  \dimm \Z_{k+\half,n}^{\mu^{\ast}}, \quad & \text{if $\mu > \mu^{\ast}$,} \\
\dimm \WZ_{k+\half,n}^{\mu^{+}} = \dimm \WZ_{k+\half,n}^{\mu^{-}} =\dimm \Z_{k+\half,n}^{\mu}, \quad & \text{if $\mu = \mu^{\ast}$,} \\
\end{array}
$$
where $\dimm \Z_{k+\half,n}^\mu$ and $\dimm \Z_{k+\half,n}^{\mu^{\ast}}$ are given by the formula in Theorem \ref{T:cent}\,(b).
 \end{itemize}
 \end{thm}  

\bigskip
 \noindent
The next corollary gives some particular instances of Theorem \ref{T:Ardims} of  special interest.

 \begin{cor}\label{C:Ardims}  Assume $k \in \mathbb{Z}_{\geq 0}$ and $r \ge 2$.  Recall the definitions of the Bell numbers $\mathsf{B}(k,n)$ and $\mathsf{B}(k)$ from
  \eqref{eq:Bell1} and  \eqref{eq:Bell2}.
 \begin{itemize} 
\item[{\rm (a)}] 
$\displaystyle{\dimm \WZ_{k,n}^{[n]} = \dimm \Z_{k,n}^{[n]} + \dimm \Z_{k,n}^{[1^n]}= \sum_{t =0 }^{n}\lb{k \atop t}\rb \, + \, \lb{k \atop n-1}\rb + \lb{k \atop n}\rb}$ \\
$\phantom{.}\hskip.535in\displaystyle{ = \mathsf{B}(k,n) + \, \lb{k \atop n-1}\rb + \lb{k \atop n}\rb.}$

\item[{\rm (b)}] $\displaystyle{\dimm \WZ_{k}(n) =  \dimm \WZ_{2k,n}^{[n]} = \mathsf{B}(2k,n) + \, \lb{2k \atop n-1}\rb + \lb{2k \atop n}\rb}$.

In particular, $\dimm \WZ_{k}(n) = \mathsf{B}(2k) + 1$ if $n = 2k+1$, and $\dimm \WZ_{k}(n) = \mathsf{B}(2k)$ if $n > 2k+1$. 
 
\item[{\rm (c)}]   $\displaystyle{\dimm \WZ_{k+\half,n}^{[n-1]} =    \dimm \Z_{k+\half,n}^{[n-1]} +
  \dimm \WZ_{k+\half,n}^{[1^{n-1}]} =
    \sum_{j = 1}^{n}  \lb{k+1\atop j}\rb
 + \   \lb{k+1 \atop  n-1}\rb + \lb{k+1 \atop  n}\rb}$ \\ 
$\phantom{.}\hskip.70in\displaystyle{ =  \mathsf{B}(k+1,n) +  \   \lb{k+1 \atop  n-1}\rb + \lb{k+1 \atop  n}\rb 
= \dimm \WZ_{k+1,n}^{[n]}.}$ 

\item[{\rm (d)}] $\displaystyle{ \dimm \WZ_{k+\half}(n) =  \dimm \WZ_{2k+\half,n}^{[n-1]} = \mathsf{B}(2k+1,n) +    \lb{2k+1 \atop  n-1}\rb + \lb{2k+1 \atop  n}\rb.}$ 

In particular, $\dimm \WZ_{k+\half}(n) = \mathsf{B}(2k+1)+1$ if $n = 2k +2$, and $\dimm \WZ_{k+\half}(n)  = \mathsf{B}(2k+1)$ if $n > k +2$.

\end{itemize}
 \end{cor}     
 \medskip   
 
 \begin{remark}  Part (b) of Corollary \ref{C:Ardims} was shown by Bloss \cite{Bl1,Bl2}  by different 
 methods.  Part (d) extends that result to the centralizer algebras  $\WZ_{k+\half}(n)$ and gives
 some indication of how the algebras $\WZ_{k+\half}(n)$ ``fill the gap'' between the
 integer levels.  \end{remark} 
 
\begin{example}   Corollary \ref{C:Ardims} (c) says that  for $k=3$ and $n=4$,   
$$ \dimm \WZ_{3+\half,4}^{[3]} = \sum_{j = 1}^{4}  \lb{4\atop j}\rb
 + \   \lb{4 \atop  3}\rb + \lb{4 \atop  4}\rb  = 1+7+2 (6+1) = 22.$$  
This is the subscript on the partition $[3]$ in the last row of the Bratteli diagram in Figure 2.\end{example}
 
 \section{The centralizer algebra $\QZ_k(n): = \End_{\S_n}(\Rs_n^{\ot k})$\, for \, $\Rs_n =  \S_n^{[n-1,1]}$ \\  and
 its relatives\label{section:quasi}}

 In \cite{DO},  Daugherty and Orellana investigated the centralizer algebra $\QZ_k(n): = \End_{\S_n} (\Rs_n^{\ot k})$,
 where $\Rs_n= \S_n^{[n-1,1]}$, and proved that there is a variant of the partition algebra, that they termed the 
 \emph{quasi-partition algebra} and denoted $\mathsf{QP}_k(n)$. They  exhibited an algebra homomorphism $\mathsf{QP}_k(n) \rightarrow \End_{\S_n} (\Rs_n^{\ot k})$ and  showed that this mapping is always a surjection and is an isomorphism when $n \ge 2k$.    The irreducible modules 
 $\QZ_{k,n}^\lam$ for
 $\QZ_k(n)$ are indexed by partitions $\lam = [\lam_1, \lam_2, \dots, \lam_n]  \vdash n$.   
 
 In this last section, we determine a formula for the 
 dimensions of these irreducible modules.    The dimension expression we obtain holds
 for arbitrary values of $k$ and $n$ and differs from that in \cite[Thm.~4.6]{DO}, which is valid for  $n > k+\lam_2$,  and is more closely
 related to the one in \cite[Cor.~2.2]{D}, which holds for $n \ge  k+\lam_2$.  
We also extend these results to the case of the corresponding centralizer algebra of the alternating group:
   $\widehat\QZ_k(n): = \End_{\A_n} (\Rs_n^{\ot k})$.
  
 We adopt the  notation in Table \ref{table2} for various centralizer algebras and their irreducible modules
 associated with $\Rs_n^{\ot k}$.    In this table,  for all $k \in \ZZ_{\ge 0}$,   $\mathsf{q}\Lambda_{k,\S_n}$  (resp. $\mathsf{q} \Lambda_{k,\A_n}$)  is
the set of indices for the irreducible $\QZ_k(n)$-summands   
(resp. $\widehat \QZ_k(n)$-summands) in $\Rs_n^{\ot k}$ with multiplicity at least one;  similarly      
  $\mathsf{q}\Lambda_{k,\S_{n-1}}$  (resp. $\mathsf{q}\Lambda_{k,\A_{n-1}}$)  is
the set of indices for the irreducible $\QZ_{k+\half}(n)$-summands   
(resp. $\widehat \QZ_{k+\half}(n)$-summands) in $\Rs_n^{\ot k}$ with multiplicity at least one.

\begin{table}[h!]
\centering
\begin{tabular}[t]{|c|c|}
\hline
        centralizer algebra  &  irreducible modules \\ \hline \hline  
 $\QZ_k(n): = \End_{\S_n} (\Rs_n^{\ot k})\phantom{\bigg\vert} \!\!$  &  $\QZ_{k,n}^\lam, \ \  \lam \vdash n$, \ \  $\lam \in \mathsf{q}\Lambda_{k,\S_n} \subseteq \Lambda_{\S_n}$  \\ \hline
 $\QZ_{k+\half}(n): = \End_{\S_{n-1}} (\Rs_n^{\ot k})\phantom{\bigg\vert} \!\!$  &  $\QZ_{k+\half,n}^\mu, \ \  \mu \vdash n-1$ \ \  $\mu \in \mathsf{q}\Lambda_{k+\half,\S_{n-1}} \subseteq \Lambda_{\S_{n-1}}$ \\ \hline 
  $\widehat{\QZ}_k(n): = \End_{\A_n} (\Rs_n^{\ot k})\phantom{\bigg\vert} \!\!$  &  $\widehat\QZ_{k,n}^\lam, \ \  \lam \vdash n$, \ $\lam > \lam^{\ast}$,
  \ \ $\lam \in \mathsf{q}\Lambda_{k,\A_n} \subseteq \Lambda_{\A_n}$  \\  
  & $\widehat \QZ_{k,n}^{\lam^\pm}, \ \  \lam \vdash n$, \ $\lam = \lam^{\ast}$,   \ \ $\lam \in \mathsf{q}\Lambda_{k,\A_n} \subseteq \Lambda_{\A_n}$   \\  \hline
  $\widehat{\QZ}_{k+\half}(n): = \End_{\A_{n-1}} (\Rs_n^{\ot k})\phantom{\bigg\vert} \!\!$  &  $\widehat\QZ_{k+\half,n}^\mu, \ \  \mu \vdash n-1$, \ $\mu> \mu^{\ast}$,
    \ \ 
    \!\!\!$\mu \in \mathsf{q}\Lambda_{k+\half,\A_{n-1}} \subseteq \Lambda_{\A_{n-1}}\!\!$    \\  
  & $\widehat \QZ_{k+\half,n}^{\mu^\pm}, \ \  \mu \vdash n$, \ $\mu = \mu^{\ast}$, \ \ \phantom{\bigg\vert} $\mu \in \mathsf{q}\Lambda_{k+\half,\A_{n-1}} \subseteq \Lambda_{\A_{n-1}}$  \\  \hline
	\end{tabular}
	\caption{Notation for the centralizer algebras and modules associated with the tensor  product $\mathsf{R}_n^{\ot k}$ of the reflection module $\mathsf{R}_n = \mathsf{S}_n^{[n-1,1]}$ of $\S_n$ and its restriction to $\mathsf{S}_{n-1}, \A_n,$ and
	$\mathsf{A}_{n-1}$. \label{table2}}
	\end{table} 
 
The permutation module $\mathsf{M}_n$ of the symmetric group satisfies $\mathsf{M}_n \cong  \mathsf{R}_n \oplus  \S_n^{[n]}$, where $\Rs_n= \S_n^{[n-1,1]}$ is the $(n-1)$-dimensional reflection representation of $\S_n$ and $ \S_n^{[n]}$ is the trivial module. Applying Proposition \ref{prop:QuasiBinomial} (a) and Theorem \ref{T:Ardims} gives the following:
\newpage 
 
\begin{thm}\label{T:QPcent} Let $k,n \in \ZZ_{\ge 0}$ with $ n \ge 1$. The dimensions of the irreducible modules for $\QZ_k(n)$,  $\QZ_{k+\half}(n)$,
$\widehat\QZ_k(n)$ and  $\widehat\QZ_{k+\half}(n)$ are as follows (using notation from Tables \ref{table1} and \ref{table2}).
 
\begin{itemize}
 \item [{\rm (a)}] For $\lam \vdash n$, $\lam \in \mathsf{q}\Lambda_{k,\S_n}$,  \hskip.6in
$\displaystyle{ \dimm \QZ_{k,n}^\lam =\sum_{\ell=0}^k (-1)^{k-\ell} {k \choose \ell}  \dimm \Z_{\ell,n}^\lam}.$

\item [{\rm (b)}] For $\mu \vdash n-1$, $\mu \in \mathsf{q}\Lambda_{k, \S_{n-1}}$,  \hskip.2in
$\displaystyle{\dimm \QZ_{k+\half,n}^\mu  =  \sum_{\ell=0}^k  (-1)^{k-\ell} {k \choose \ell}  \dimm \Z_{\ell+\half,n}^\mu}.$
 
\item[{\rm (c)}] For $\lam \vdash n$ with  $\lam \in \mathsf{q}\Lambda_{k, \A_n}$,  \quad
\begin{equation*}
\hskip-.2in
\begin{array}{ll}
\displaystyle{\dimm \widehat \QZ_{k,n}^{\lam} 
= \sum_{\ell=0}^k (-1)^{k-\ell} {k \choose \ell} \dimm \WZ_{\ell,n}^{\lam}
= \sum_{\ell=0}^k (-1)^{k-\ell} {k \choose \ell}\!\!\left( \dimm \Z_{\ell,n}^{\lam} + \dimm \Z_{\ell,n}^{\lam^\ast} \right),} 
& \text{if $\lam > \lam^{\ast},$}\\
\displaystyle{\dimm \widehat \QZ_{k,n}^{\lam^\pm} 
= \sum_{\ell=0}^k (-1)^{k-\ell} {k \choose \ell} \dimm \WZ_{\ell,n}^{\lam^\pm}
= \sum_{\ell=0}^k (-1)^{k-\ell} {k \choose \ell} \dimm \Z_{\ell,n}^{\lam} = \dimm \QZ_{k,n}^{\lam},} & \text{if $\lam = \lam^{\ast}.$}\\
\end{array}
\end{equation*}

 \item[{\rm (d)}] For $\mu \vdash n-1$ with $\mu \in \mathsf{q}\Lambda_{k+\half, \A_{n-1}}$, 
 \begin{equation*}
\begin{array}{llll}
\dimm \widehat \QZ_{k+\half,n}^{\mu} 
&=& \displaystyle{\sum_{\ell=0}^k (-1)^{k-\ell} {k \choose \ell} \dimm \WZ_{\ell+\half,n}^{\mu}}\\
&=& \displaystyle{\sum_{\ell=0}^k (-1)^{k-\ell} {k \choose \ell}\left( \dimm \Z_{\ell+\half,n}^{\mu} + \dimm \Z_{\ell+\half,n}^{\mu^\ast} \right),} 
& \text{if $\mu > \mu^{\ast}$,}\\ 
\\
\dimm \widehat \QZ_{k+\half,n}^{\mu^\pm} 
&=& \displaystyle{\sum_{\ell=0}^k (-1)^{k-\ell} {k \choose \ell} \dimm \WZ_{\ell+\half,n}^{\mu^\pm}} \\
&=& \displaystyle{\sum_{\ell=0}^k (-1)^{k-\ell} {k \choose \ell} \dimm \Z_{\ell+\half,n}^{\mu} = \dimm  \QZ_{k+\half,n}^{\mu},} 
& \text{if $\mu = \mu^{\ast}$.}\\
\end{array}
\end{equation*}
 \end{itemize} 
\end{thm}

\noindent Applying Proposition  \ref{prop:QuasiBinomial}(b) and Theorem \ref{T:QPcent} gives the following:

 \begin{cor}\label{cor:QPcent} Let $k,n \in \ZZ_{\ge 0}$ with $n > 0$,  and let the notation be as in Table \ref{table2}.
  \begin{itemize}
 \item [{\rm (a)}] $\displaystyle{
  \dimm \QZ_k(n)  =\dimm \QZ_{2k,n}^{[n]} = \sum_{\ell=0}^{2k} (-1)^{2k-\ell}{2k \choose \ell} \mathsf{B}(\ell,n)}$\\
$\displaystyle{\hspace{1.7cm}  \left(= \sum_{\ell=0}^{2k} (-1)^{2k-\ell}{2k \choose \ell}   \mathsf{B}(\ell) = 1+ \sum_{\ell=1}^{2k} (-1)^{\ell-1} \mathsf{B}(2k-\ell) \quad  \hbox{if  \, \ $n \ge 2k+2$}\right)}.$

\item [{\rm (b)}] $\displaystyle{
\dimm \QZ_{k+\half}(n) \ = \ \dimm \QZ_{2k+\half,n}^{[n-1]} = \sum_{\ell=0}^{2k} (-1)^{2k-\ell}{2k \choose \ell} \mathsf{B}(\ell+1,n)}$\\
$ \displaystyle{ \hspace{2.2cm} \left( = \ \sum_{\ell=0}^{2k} (-1)^{2k-\ell}{2k \choose \ell}   \mathsf{B}(\ell+1)  = \mathsf{B}(2k)  \quad    \hbox{if \ \, $n \geq 2k+1$}\right)}.$

\item[{\rm (c)}] $\displaystyle{\dimm  \widehat \QZ_k(n) =\dimm\widehat \QZ_{2k,n}^{[n]} \ =\   \sum_{\ell=0}^{2k} (-1)^{2k-\ell} {2k \choose \ell} \left( \mathsf{B}(\ell,n)+ 
  \lb{\ell \atop n-1}\rb + \lb{\ell \atop n}\rb\right)}$ \\ 
$\displaystyle{\hspace{1.8cm}  \left(=\  \sum_{\ell=0}^{2k} (-1)^{2k-\ell}{2k \choose \ell}  \mathsf{B}(\ell)  
  = 1+ \sum_{\ell=1}^{2k} (-1)^{\ell-1} \mathsf{B}(2k-\ell) \quad  \hbox{ if \ \, $n \ge 2k+2$}\right).}$

 \item[{\rm (d)}]  $\displaystyle{\dimm  \widehat \QZ_{k+\half}(n)= \ \dimm\widehat \QZ_{2k+\half,n}^{[n-1]}}  $\\
$\displaystyle{ \hspace{2.4cm}     =\   \sum_{\ell=0}^{2k} (-1)^{2k-\ell} {2k \choose \ell} \left( \mathsf{B}(\ell+1,n)+ 
  \lb{\ell +1 \atop n-1}\rb + \lb{\ell+1 \atop n}\rb\right)}$ \\ 
$\displaystyle{ \hspace{2.1cm}   \left(= \  \sum_{\ell=0}^{2k} (-1)^{2k-\ell}{2k \choose \ell}  \mathsf{B}(\ell+1)  = \mathsf{B}(2k) 
\quad  \hbox{  if \ \, $n \ge 2k+2$}\right).}$
  \end{itemize} 
\end{cor}

 \begin{proof}    The first two equalities in parts (a)-(d) of the corollary follow from Proposition \ref{prop:QuasiBinomial}\,(b), \eqref{eq:trivial}, \eqref{eq:halftriv}, 
 and Corollary \ref{C:Ardims}\,(b) and (d).
The final equality in (a) and (c) can be seen as follows:      Let  $\upsilon_\ell$ be the number
of set partitions of $\{1,\dots, \ell\}$, with no blocks of size 1.      Then, as shown in \cite[Sec.~3.5]{Be},
$\upsilon_{\ell} + \upsilon_{\ell+1} = \mathsf{B}(\ell)$, (the $\ell$th Bell number).    Now  \cite[Sec.~1]{SW} implies that
$\upsilon_{2k} = \sum_{\ell=0}^{2k}(-1)^{2k-\ell}{2k \choose {  \ell}} \mathsf{B}(\ell)$.    However, substituting the expression 
$\upsilon_{\ell} + \upsilon_{\ell+1}$  for $\mathsf{B}(\ell)$ shows that 
the telescoping sum $1+ \sum_{\ell=1}^{2k} (-1)^{ \ell-1} \mathsf{B}(2k-\ell) = \upsilon_{2k}$
also.   Hence, the two expressions for $\upsilon_{2k}$  equal.  The final equality in parts  (b) and (d) is a well-known property of Bell numbers (see for example \cite[(1.2)]{SW}).   \end{proof}

\begin{remark}  The result from Corollary \ref{cor:QPcent}  (a) that  $\dimm \QZ_k(n) =  1+\sum_{\ell=1}^{2k} (-1)^{\ell-1} \mathsf{B}(2k-\ell)$ when $n \ge 2k+2$ was shown in \cite[Cor.~2.6]{DO}.      As noted there,  the sequence $\{\upsilon_\ell\}$ is  \#A000296
in  \cite{OEIS}.   \end{remark}    

\noindent \ \ \ The results in Theorem \ref{T:QPcent} enable us to conclude the following for generic quasi-partition algebras.

\begin{cor} \label{C:QPdims}   Let  $\nu$ be a partition with $0 \leq \vert \nu \vert \leq k$.   For $\para \not \in 
\{0,1,\ldots, 2k-1\}$,  let  $\QP_{k,\para}^\nu$ denote the irreducible $\QP_k(\para)$-module.   Then
\begin{eqnarray*}
\dimm \QP_{k,\para}^\nu &=& f^{\nu} \sum_{\ell=0}^k (-1)^{k-\ell}{k \choose \ell} \left(\sum_{t =\vert \nu\vert }^{\ell} {t \choose {\vert \nu\vert}} \lb{\ell \atop t}\rb\right) \qquad  \big(= f^\nu \ \ \hbox{when} \ \
 |\nu| = k\big).
\end{eqnarray*} \end{cor} 

The Bratteli diagrams
constructed using the reflection  module $\Rs_n$  for the pairs $(\S_n,\S_{n-1})$, $(\A_n, \A_{n-1})$ for $n=6$
are displayed in \ref{sec:QBratteliS6S5} and \ref{sec:QBratteliA6A5} of the Appendix.   The subscript on a partition  at level $\ell \in \half \ZZ_{\ge 0}$ is the dimension
of the irreducible module for the centralizer algebra $\QZ_\ell(6)$.   For $k \in \ZZ_{\ge 0}$,    $\Ind_{\S_{n-1}}^{\S_n}\Res_{\S_{n-1}}^{\S_n}(\Rs_n^{\ot k})$ is isomorphic as
an $\S_n$-module to $\Rs_n^{\ot k} \oplus \Rs_n^{\ot (k+1)}$.    This implies that the subscripts on level $k +\half$ are gotten
from  level $k$ by Pascal addition;  however, the 
 subscripts  on level $k+1$ are obtained by 
 first performing Pascal addition from level  $k+\half$ and then subtracting the corresponding subscript  
from level $k$. 
\medskip


 \newpage  
 
  \appendix
  \section{Appendix: Bratteli Diagrams} 
  
   \subsection{\bf Levels  \ $\ell = 0,\half,1,\ldots,\frac{7}{2},4$ \, of the Bratteli diagram $\mathcal{B}(\S_6,\S_5)$} \label{sec:BratteliS6S5}

\Yboxdim{6pt}
\Ylinethick{.6pt}
\noindent  

\noindent Level  $\ell = \frac{7}{2}$ is  the first time the centralizer algebra loses  a dimension from 
the generic dimension, which is the 7th Bell number   $\mathsf{B}(7) =  877$.
\bigskip

$$
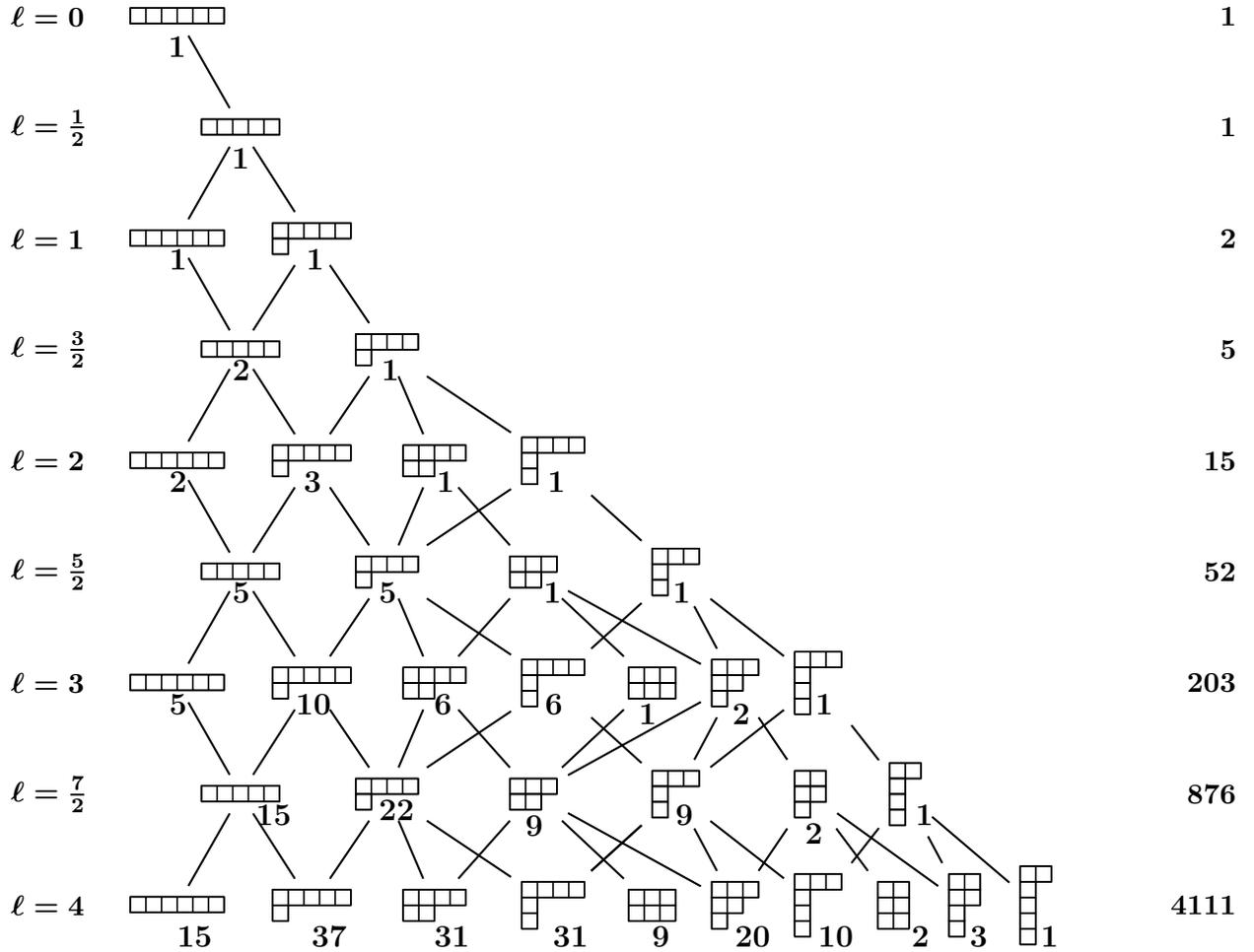
\begin{tikzpicture}[line width=.5pt,xscale=0.16,yscale=0.30]
\path (-10,0)  node[anchor=west]  {$\bsym{\ell=0}$};
\path (0,0)  node[anchor=west] (S6-0) {\yng(6)};
\draw (S6-0) node[below=4pt, black] {\nmm 1};
\path (-10,-5)  node[anchor=west]  {$\bsym{\ell=\frac{1}{2}}$};
\path (6,-5)  node[anchor=west] (S5-1) {\yng(5)};
\draw (S5-1) node[below=4pt,black] {\nmm 1};
\path (-10,-10)  node[anchor=west]  {$\bsym{\ell=1}$};
\path (0,-10)  node[anchor=west] (S6-2) {\yng(6)};
\path (12,-10)  node[anchor=west] (S51-2) {\yng(5,1)};
\draw (S6-2) node[below=8pt,right=-7pt,black] {\nmm 1};
\draw (S51-2) node[below=8pt, right=-6pt,black] {\nmm 1};
\path (-10,-15)  node[anchor=west]  {$\bsym{\ell=\frac{3}{2}}$};
\path (6,-15)  node[anchor=west] (S5-3) {\yng(5)};
\path (19,-15)  node[anchor=west] (S41-3) {\yng(4,1)};
\draw (S5-3) node[below=8pt, right=-7pt,black] {\nmm 2};
\draw (S41-3) node[below=8pt, right=-6pt,black] {\nmm 1};
\path (-10,-20)  node[anchor=west]  {$\bsym{\ell=2}$};
\path (0,-20)  node[anchor=west] (S6-4) {\yng(6)};
\path (12,-20)  node[anchor=west] (S51-4) {\yng(5,1)};
\path (23,-20)  node[anchor=west] (S42-4) {\yng(4,2)};
\path (33,-20)  node[anchor=west] (S411-4) {\yng(4,1,1)};
\draw (S6-4)    node[below=8pt, right=-7pt,black]{\nmm 2};
\draw (S51-4)  node[below=8pt, right=-7pt,black] {\nmm 3};
\draw (S42-4)  node[below=8pt, right=-3pt,black] {\nmm 1};
\draw (S411-4) node[below=8pt, right=-6pt,black] {\nmm 1};
\path (-10,-25)  node[anchor=west]  {$\bsym{\ell=\frac{5}{2}}$};
\path (6,-25)  node[anchor=west] (S5-5) {\yng(5)};
\path (19,-25)  node[anchor=west] (S41-5) {\yng(4,1)};
\path (32,-25)  node[anchor=west] (S32-5) {\yng(3,2)};
\path (44,-25)  node[anchor=west] (S311-5) {\yng(3,1,1)};
\draw (S5-5) node[below=8pt, right=-7pt,black]{\nmm 5};
\draw (S41-5) node[below=8pt, right=-7pt,black] {\nmm 5};
\draw (S32-5) node[below=8pt, right=0pt,black] {\nmm 1};
\draw (S311-5) node[below=8pt, right=-5pt,black] {\nmm 1};
\path (-10,-30)  node[anchor=west]  {$\bsym{\ell=3}$};
\path (0,-30)  node[anchor=west] (S6-6) {\yng(6)};
\path (12,-30)  node[anchor=west] (S51-6) {\yng(5,1)};
\path (23,-30)  node[anchor=west] (S42-6) {\yng(4,2)};
\path (33,-30)  node[anchor=west] (S411-6) {\yng(4,1,1)};
\path (42,-30)  node[anchor=west] (S33-6) {\yng(3,3)};
\path (49,-30)  node[anchor=west] (S321-6) {\yng(3,2,1)};
\path (56,-30)  node[anchor=west] (S3111-6) {\yng(3,1,1,1)};
\draw (S6-6)      node[below=8pt, right=-7pt,black]{\nmm 5};
\draw (S51-6)    node[below=8pt, right=-10pt,black] {\nmm 10};
\draw (S42-6)    node[below=8pt, right=-4pt,black] {\nmm 6};
\draw (S411-6)   node[below=8pt, right=-7pt,black]  {\nmm 6};
\draw (S33-6) node[below=12pt, right=-9pt,black] {\nmm 1};
\draw (S321-6) node[below=12pt, right=-5pt,black] {\nmm 2};
\draw (S3111-6) node[below=8pt, right=-5pt,black] {\nmm 1};
\path (-10,-35)  node[anchor=west]  {$\bsym{\ell=\frac{7}{2}}$};
\path (6,-35)  node[anchor=west] (S5-7) {\yng(5)};
\path (19,-35)  node[anchor=west] (S41-7) {\yng(4,1)};
\path (32,-35)  node[anchor=west] (S32-7) {\yng(3,2)};
\path (44,-35)  node[anchor=west] (S311-7) {\yng(3,1,1)};
\path (56,-35)  node[anchor=west] (S221-7) {\yng(2,2,1)};
\path (64,-35)  node[anchor=west] (S2111-7) {\yng(2,1,1,1)};
\draw (S5-7)  node[below=8pt, right=2pt,black] {\nmm 15};
\draw (S41-7)  node[below=6pt, right=-7pt,black] {\nmm 22};
\draw (S32-7)  node[below=12pt, right=-7pt,black] {\nmm 9};
\draw (S311-7)  node[below=8pt, right=-4pt,black] {\nmm 9};
\draw (S221-7)  node[below=15pt, right=-6pt,black]{\nmm 2};
\draw (S2111-7)  node[below=8pt, right=0pt,black]{\nmm 1};
\path (-10,-40)  node[anchor=west]  {$\bsym{\ell=4}$};
\path (0,-40)  node[anchor=west] (S6-8) {\yng(6)};
\path (12,-40)  node[anchor=west] (S51-8) {\yng(5,1)};
\path (23,-40)  node[anchor=west] (S42-8) {\yng(4,2)};
\path (33,-40)  node[anchor=west] (S411-8) {\yng(4,1,1)};
\path (42,-40)  node[anchor=west] (S33-8) {\yng(3,3)};
\path (49,-40)  node[anchor=west] (S321-8) {\yng(3,2,1)};
\path (56,-40)  node[anchor=west] (S3111-8) {\yng(3,1,1,1)};
\path (63,-40)  node[anchor=west] (S222-8) {\yng(2,2,2)};
\path (69,-40)  node[anchor=west] (S2211-8) {\yng(2,2,1,1)};
\path (75,-40)  node[anchor=west] (S21111-8) {\yng(2,1,1,1,1)};
\draw (S6-8)          node[below=12pt, right=-4pt,black]  {\nmm 15};
\draw (S51-8)        node[below=12pt, right=-4pt,black] {\nmm 37};
\draw (S42-8)        node[below=12pt, right=-4pt,black] {\nmm 31};
\draw (S411-8)      node[below=12pt, right=-4pt,black] {\nmm 31};
\draw (S33-8)        node[below=12pt, right=-4pt,black] {\nmm 9};
\draw (S321-8)      node[below=12pt, right=-4pt,black] {\nmm 20};
\draw (S3111-8)    node[below=12pt, right=-4pt,black] {\nmm 10};
\draw (S222-8)     node[below=12pt, right=3pt,black] {\nmm 2};
\draw (S2211-8)    node[below=12pt, right=-2pt,black] {\nmm 3};
\draw (S21111-8)  node[below=12pt, right=-2pt,black] {\nmm 1};
\draw (95,0) node[anchor=east,black] {$~~{\mathbf{1}}$};
\draw (95,-5) node[anchor=east,black] {$~~{\mathbf{1}}$};
\draw (95,-10) node[anchor=east,black] {$~~{\mathbf{2}}$};
\draw (95,-15) node[anchor=east,black] {$~~{\mathbf{5}}$};
\draw (95,-20) node[anchor=east,black] {$~~{\mathbf{15}}$};
\draw (95,-25) node[anchor=east,black] {$~~{\mathbf{52}}$};
\draw (95,-30) node[anchor=east,black] {$~~{\mathbf{203}}$};
\draw (95,-35) node[anchor=east,black] {$~~{\mathbf{876}}$};
\draw (95,-40) node[anchor=east,black] {$~~{\mathbf{4111}}$};

\path  (S6-0) edge[black,thick] (S5-1);
\path  (S6-2) edge[black,thick] (S5-1);
\path  (S51-2) edge[black,thick] (S5-1);
\path  (S6-2) edge[black,thick] (S5-3);
\path  (S51-2) edge[black,thick] (S5-3);
\path  (S51-2) edge[black,thick] (S41-3);
\path  (S6-4) edge[black,thick] (S5-3);
\path  (S51-4) edge[black,thick] (S5-3);
\path  (S51-4) edge[black,thick] (S41-3);
\path  (S42-4) edge[black,thick] (S41-3);
\path  (S411-4) edge[black,thick] (S41-3);
\path  (S6-4) edge[black,thick] (S5-5);
\path  (S51-4) edge[black,thick] (S5-5);
\path  (S51-4) edge[black,thick] (S41-5);
\path  (S42-4) edge[black,thick] (S41-5);
\path  (S42-4) edge[black,thick] (S32-5);
\path  (S411-4) edge[black,thick] (S41-5);
\path  (S411-4) edge[black,thick] (S311-5);
\path  (S6-6) edge[black,thick] (S5-5);
\path  (S51-6) edge[black,thick] (S5-5);
\path  (S51-6) edge[black,thick] (S41-5);
\path  (S42-6) edge[black,thick] (S41-5);
\path  (S42-6) edge[black,thick] (S32-5);
\path  (S411-6) edge[black,thick] (S41-5);
\path  (S411-6) edge[black,thick] (S311-5);
\path  (S33-6) edge[black,thick] (S32-5);
\path  (S321-6) edge[black,thick] (S32-5);
\path  (S321-6) edge[black,thick] (S311-5);
\path  (S3111-6) edge[black,thick] (S311-5);
\path  (S6-6) edge[black,thick] (S5-7);
\path  (S51-6) edge[black,thick] (S5-7);
\path  (S51-6) edge[black,thick] (S41-7);
\path  (S42-6) edge[black,thick] (S41-7);
\path  (S42-6) edge[black,thick] (S32-7);
\path  (S33-6) edge[black,thick] (S32-7);
\path  (S321-6) edge[black,thick] (S32-7);
\path  (S321-6) edge[black,thick] (S311-7);
\path  (S321-6) edge[black,thick] (S221-7);
\path  (S411-6) edge[black,thick] (S41-7);
\path  (S411-6) edge[black,thick] (S311-7);
\path  (S3111-6) edge[black,thick] (S311-7);
\path  (S3111-6) edge[black,thick] (S2111-7);
\path  (S6-8) edge[black,thick] (S5-7);
\path  (S51-8) edge[black,thick] (S5-7);
\path  (S51-8) edge[black,thick] (S41-7);
\path  (S42-8) edge[black,thick] (S41-7);
\path  (S411-8) edge[black,thick] (S41-7);
\path  (S42-8) edge[black,thick] (S32-7);
\path  (S33-8) edge[black,thick] (S32-7);
\path  (S321-8) edge[black,thick] (S32-7);
\path  (S411-8) edge[black,thick] (S311-7);
\path  (S321-8) edge[black,thick] (S311-7);
\path  (S321-8) edge[black,thick] (S221-7);
\path  (S222-8) edge[black,thick] (S221-7);
\path  (S2211-8) edge[black,thick] (S221-7);
\path  (S411-8) edge[black,thick] (S311-7);
\path  (S3111-8) edge[black,thick] (S311-7);
\path  (S3111-8) edge[black,thick] (S2111-7);
\path  (S2211-8) edge[black,thick] (S2111-7);
\path  (S21111-8) edge[black,thick] (S2111-7);
\end{tikzpicture}
$$
\newpage  

 \subsection{Levels  $\ell = 0,\half,1,\ldots,\frac{7}{2},4$ of the quasi-Bratteli diagram
$ \mathcal{QB}(\S_6,\S_5)$}  \label{sec:QBratteliS6S5}
 
\Yboxdim{6pt}
\Ylinethick{.6pt} 
\bigskip

To calculate the subscripts on the half-integer rows, use Pascal addition of the subscripts from the row above. To calculate the subscripts on integer level rows, first use Pascal addition from the row above,  and then subtract the subscript on the same partition from two rows above.
\medskip
  
$$
\begin{tikzpicture}[line width=.5pt,xscale=0.14,yscale=0.30]
\path (-10,0)  node[anchor=west]  {$\bsym{\ell=0}$};
\path (0,0)  node[anchor=west] (S6-0) {\yng(6)};
\draw (S6-0) node[below=4pt, black] {\nmm 1};
\path (-10,-5)  node[anchor=west]  {$\bsym{\ell=\frac{1}{2}}$};
\path (6,-5)  node[anchor=west] (S5-1) {\yng(5)};
\draw (S5-1) node[below=4pt,black] {\nmm 1};
\path (-10,-10)  node[anchor=west]  {$\bsym{\ell=1}$};
\path (0,-10)  node[anchor=west] (S6-2) {\yng(6)};
\path (12,-10)  node[anchor=west] (S51-2) {\yng(5,1)};
\draw (S6-2) node[below=8pt,right=-7pt,black] {\nmm 0};
\draw (S51-2) node[below=8pt, right=-6pt,black] {\nmm 1};
\path (-10,-15)  node[anchor=west]  {$\bsym{\ell=\frac{3}{2}}$};
\path (6,-15)  node[anchor=west] (S5-3) {\yng(5)};
\path (19,-15)  node[anchor=west] (S41-3) {\yng(4,1)};
\draw (S5-3) node[below=8pt, right=-7pt,black] {\nmm 1};
\draw (S41-3) node[below=8pt, right=-6pt,black] {\nmm 1};
\path (-10,-20)  node[anchor=west]  {$\bsym{\ell=2}$};
\path (0,-20)  node[anchor=west] (S6-4) {\yng(6)};
\path (12,-20)  node[anchor=west] (S51-4) {\yng(5,1)};
\path (23,-20)  node[anchor=west] (S42-4) {\yng(4,2)};
\path (33,-20)  node[anchor=west] (S411-4) {\yng(4,1,1)};
\draw (S6-4)    node[below=8pt, right=-7pt,black]{\nmm 1};
\draw (S51-4)  node[below=8pt, right=-7pt,black] {\nmm 1};
\draw (S42-4)  node[below=8pt, right=-3pt,black] {\nmm 1};
\draw (S411-4) node[below=8pt, right=-6pt,black] {\nmm 1};
\path (-10,-25)  node[anchor=west]  {$\bsym{\ell=\frac{5}{2}}$};
\path (6,-25)  node[anchor=west] (S5-5) {\yng(5)};
\path (19,-25)  node[anchor=west] (S41-5) {\yng(4,1)};
\path (32,-25)  node[anchor=west] (S32-5) {\yng(3,2)};
\path (44,-25)  node[anchor=west] (S311-5) {\yng(3,1,1)};
\draw (S5-5) node[below=8pt, right=-7pt,black]{\nmm 2};
\draw (S41-5) node[below=8pt, right=-7pt,black] {\nmm 3};
\draw (S32-5) node[below=8pt, right=0pt,black] {\nmm 1};
\draw (S311-5) node[below=8pt, right=-5pt,black] {\nmm 1};
\path (-10,-30)  node[anchor=west]  {$\bsym{\ell=3}$};
\path (0,-30)  node[anchor=west] (S6-6) {\yng(6)};
\path (12,-30)  node[anchor=west] (S51-6) {\yng(5,1)};
\path (23,-30)  node[anchor=west] (S42-6) {\yng(4,2)};
\path (33,-30)  node[anchor=west] (S411-6) {\yng(4,1,1)};
\path (42,-30)  node[anchor=west] (S33-6) {\yng(3,3)};
\path (49,-30)  node[anchor=west] (S321-6) {\yng(3,2,1)};
\path (56,-30)  node[anchor=west] (S3111-6) {\yng(3,1,1,1)};
\draw (S6-6)      node[below=8pt, right=-7pt,black]{\nmm 1};
\draw (S51-6)    node[below=8pt, right=-10pt,black] {\nmm 4};
\draw (S42-6)    node[below=8pt, right=-4pt,black] {\nmm 3};
\draw (S411-6)   node[below=8pt, right=-7pt,black]  {\nmm 3};
\draw (S33-6) node[below=12pt, right=-9pt,black] {\nmm 1};
\draw (S321-6) node[below=12pt, right=-5pt,black] {\nmm 2};
\draw (S3111-6) node[below=8pt, right=-5pt,black] {\nmm 1};
\path (-10,-35)  node[anchor=west]  {$\bsym{\ell=\frac{7}{2}}$};
\path (6,-35)  node[anchor=west] (S5-7) {\yng(5)};
\path (19,-35)  node[anchor=west] (S41-7) {\yng(4,1)};
\path (32,-35)  node[anchor=west] (S32-7) {\yng(3,2)};
\path (44,-35)  node[anchor=west] (S311-7) {\yng(3,1,1)};
\path (56,-35)  node[anchor=west] (S221-7) {\yng(2,2,1)};
\path (64,-35)  node[anchor=west] (S2111-7) {\yng(2,1,1,1)};
\draw (S5-7)  node[below=8pt, right=2pt,black] {\nmm 5};
\draw (S41-7)  node[below=6pt, right=-7pt,black] {\nmm 10};
\draw (S32-7)  node[below=12pt, right=-7pt,black] {\nmm 6};
\draw (S311-7)  node[below=8pt, right=-4pt,black] {\nmm 6};
\draw (S221-7)  node[below=15pt, right=-6pt,black]{\nmm 2};
\draw (S2111-7)  node[below=8pt, right=0pt,black]{\nmm 1};
\path (-10,-40)  node[anchor=west]  {$\bsym{\ell=4}$};
\path (0,-40)  node[anchor=west] (S6-8) {\yng(6)};
\path (12,-40)  node[anchor=west] (S51-8) {\yng(5,1)};
\path (23,-40)  node[anchor=west] (S42-8) {\yng(4,2)};
\path (33,-40)  node[anchor=west] (S411-8) {\yng(4,1,1)};
\path (42,-40)  node[anchor=west] (S33-8) {\yng(3,3)};
\path (49,-40)  node[anchor=west] (S321-8) {\yng(3,2,1)};
\path (56,-40)  node[anchor=west] (S3111-8) {\yng(3,1,1,1)};
\path (63,-40)  node[anchor=west] (S222-8) {\yng(2,2,2)};
\path (69,-40)  node[anchor=west] (S2211-8) {\yng(2,2,1,1)};
\path (75,-40)  node[anchor=west] (S21111-8) {\yng(2,1,1,1,1)};
\draw (S6-8)          node[below=12pt, right=-4pt,black]  {\nmm 4};
\draw (S51-8)        node[below=12pt, right=-4pt,black] {\nmm 11};
\draw (S42-8)        node[below=12pt, right=-4pt,black] {\nmm 13};
\draw (S411-8)      node[below=12pt, right=-4pt,black] {\nmm 13};
\draw (S33-8)        node[below=12pt, right=-4pt,black] {\nmm 5};
\draw (S321-8)      node[below=12pt, right=-4pt,black] {\nmm 12};
\draw (S3111-8)    node[below=12pt, right=-4pt,black] {\nmm 6};
\draw (S222-8)     node[below=12pt, right=3pt,black] {\nmm 2};
\draw (S2211-8)    node[below=12pt, right=-2pt,black] {\nmm 3};
\draw (S21111-8)  node[below=12pt, right=-2pt,black] {\nmm 1};
\draw (95,0) node[anchor=east,black] {$~~{\mathbf{1}}$};
\draw (95,-5) node[anchor=east,black] {$~~{\mathbf{1}}$};
\draw (95,-10) node[anchor=east,black] {$~~{\mathbf{1}}$};
\draw (95,-15) node[anchor=east,black] {$~~{\mathbf{2}}$};
\draw (95,-20) node[anchor=east,black] {$~~{\mathbf{4}}$};
\draw (95,-25) node[anchor=east,black] {$~~{\mathbf{15}}$};
\draw (95,-30) node[anchor=east,black] {$~~{\mathbf{41}}$};
\draw (95,-35) node[anchor=east,black] {$~~{\mathbf{202}}$};
\draw (95,-40) node[anchor=east,black] {$~~{\mathbf{694}}$};

\path  (S6-0) edge[black,thick] (S5-1);
\path  (S6-2) edge[black,thick] (S5-1);
\path  (S51-2) edge[black,thick] (S5-1);
\path  (S6-2) edge[black,thick] (S5-3);
\path  (S51-2) edge[black,thick] (S5-3);
\path  (S51-2) edge[black,thick] (S41-3);
\path  (S6-4) edge[black,thick] (S5-3);
\path  (S51-4) edge[black,thick] (S5-3);
\path  (S51-4) edge[black,thick] (S41-3);
\path  (S42-4) edge[black,thick] (S41-3);
\path  (S411-4) edge[black,thick] (S41-3);
\path  (S6-4) edge[black,thick] (S5-5);
\path  (S51-4) edge[black,thick] (S5-5);
\path  (S51-4) edge[black,thick] (S41-5);
\path  (S42-4) edge[black,thick] (S41-5);
\path  (S42-4) edge[black,thick] (S32-5);
\path  (S411-4) edge[black,thick] (S41-5);
\path  (S411-4) edge[black,thick] (S311-5);
\path  (S6-6) edge[black,thick] (S5-5);
\path  (S51-6) edge[black,thick] (S5-5);
\path  (S51-6) edge[black,thick] (S41-5);
\path  (S42-6) edge[black,thick] (S41-5);
\path  (S42-6) edge[black,thick] (S32-5);
\path  (S411-6) edge[black,thick] (S41-5);
\path  (S411-6) edge[black,thick] (S311-5);
\path  (S33-6) edge[black,thick] (S32-5);
\path  (S321-6) edge[black,thick] (S32-5);
\path  (S321-6) edge[black,thick] (S311-5);
\path  (S3111-6) edge[black,thick] (S311-5);
\path  (S6-6) edge[black,thick] (S5-7);
\path  (S51-6) edge[black,thick] (S5-7);
\path  (S51-6) edge[black,thick] (S41-7);
\path  (S42-6) edge[black,thick] (S41-7);
\path  (S42-6) edge[black,thick] (S32-7);
\path  (S33-6) edge[black,thick] (S32-7);
\path  (S321-6) edge[black,thick] (S32-7);
\path  (S321-6) edge[black,thick] (S311-7);
\path  (S321-6) edge[black,thick] (S221-7);
\path  (S411-6) edge[black,thick] (S41-7);
\path  (S411-6) edge[black,thick] (S311-7);
\path  (S3111-6) edge[black,thick] (S311-7);
\path  (S3111-6) edge[black,thick] (S2111-7);
\path  (S6-8) edge[black,thick] (S5-7);
\path  (S51-8) edge[black,thick] (S5-7);
\path  (S51-8) edge[black,thick] (S41-7);
\path  (S42-8) edge[black,thick] (S41-7);
\path  (S411-8) edge[black,thick] (S41-7);
\path  (S42-8) edge[black,thick] (S32-7);
\path  (S33-8) edge[black,thick] (S32-7);
\path  (S321-8) edge[black,thick] (S32-7);
\path  (S411-8) edge[black,thick] (S311-7);
\path  (S321-8) edge[black,thick] (S311-7);
\path  (S321-8) edge[black,thick] (S221-7);
\path  (S222-8) edge[black,thick] (S221-7);
\path  (S2211-8) edge[black,thick] (S221-7);
\path  (S411-8) edge[black,thick] (S311-7);
\path  (S3111-8) edge[black,thick] (S311-7);
\path  (S3111-8) edge[black,thick] (S2111-7);
\path  (S2211-8) edge[black,thick] (S2111-7);
\path  (S21111-8) edge[black,thick] (S2111-7);
\end{tikzpicture}
$$ 

 \newpage

 \subsection{Levels  \ $\ell = 0,\half,1,\ldots,\frac{7}{2},4$ of the  Bratteli diagram $\mathcal{B}(\A_6, \A_5)$}  \label{sec:BratteliA6A5}
 
 \bigskip
 
\Yboxdim{6pt}
\Ylinethick{.6pt}

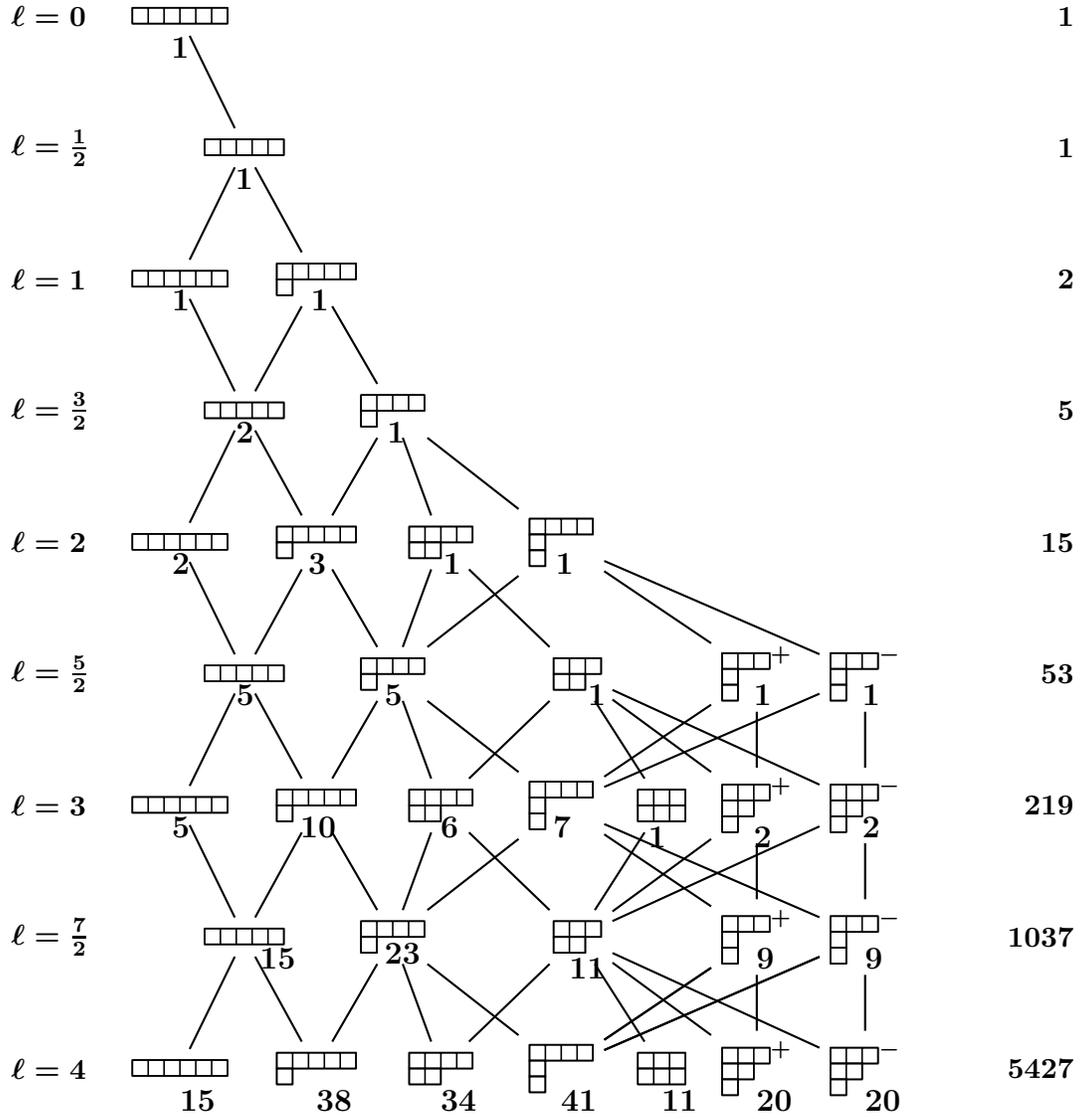
$$
\begin{tikzpicture}[line width=.5pt,xscale=0.16,yscale=0.35]
\path (-10,0)  node[anchor=west]  {$\bsym{\ell=0}$};
\path (0,0)  node[anchor=west] (S6-0) {\yng(6)};
\draw (S6-0) node[below=4pt, black] {\nmm 1};
\path (-10,-5)  node[anchor=west]  {$\bsym{\ell=\frac{1}{2}}$};
\path (6,-5)  node[anchor=west] (S5-1) {\yng(5)};
\draw (S5-1) node[below=4pt,black] {\nmm 1};
\path (-10,-10)  node[anchor=west]  {$\bsym{\ell=1}$};
\path (0,-10)  node[anchor=west] (S6-2) {\yng(6)};
\path (12,-10)  node[anchor=west] (S51-2) {\yng(5,1)};
\draw (S6-2) node[below=8pt,right=-7pt,black] {\nmm 1};
\draw (S51-2) node[below=8pt, right=-6pt,black] {\nmm 1};
\path (-10,-15)  node[anchor=west]  {$\bsym{\ell=\frac{3}{2}}$};
\path (6,-15)  node[anchor=west] (S5-3) {\yng(5)};
\path (19,-15)  node[anchor=west] (S41-3) {\yng(4,1)};
\draw (S5-3) node[below=8pt, right=-7pt,black] {\nmm 2};
\draw (S41-3) node[below=8pt, right=-6pt,black] {\nmm 1};
\path (-10,-20)  node[anchor=west]  {$\bsym{\ell=2}$};
\path (0,-20)  node[anchor=west] (S6-4) {\yng(6)};
\path (12,-20)  node[anchor=west] (S51-4) {\yng(5,1)};
\path (23,-20)  node[anchor=west] (S42-4) {\yng(4,2)};
\path (33,-20)  node[anchor=west] (S411-4) {\yng(4,1,1)};
\draw (S6-4)    node[below=8pt, right=-7pt,black]{\nmm 2};
\draw (S51-4)  node[below=8pt, right=-7pt,black] {\nmm 3};
\draw (S42-4)  node[below=8pt, right=-3pt,black] {\nmm 1};
\draw (S411-4) node[below=8pt, right=-6pt,black] {\nmm 1};
\path (-10,-25)  node[anchor=west]  {$\bsym{\ell=\frac{5}{2}}$};
\path (6,-25)  node[anchor=west] (S5-5) {\yng(5)};
\path (19,-25)  node[anchor=west] (S41-5) {\yng(4,1)};
\path (35,-25)  node[anchor=west] (S32-5) {\yng(3,2)};
\path (49,-25)  node[anchor=west] (S311p-5) {$\yng(3,1,1)^{\bsym{+}}$};
\path (58,-25)  node[anchor=west] (S311m-5) {$\yng(3,1,1)^{\bsym{-}}$};
\draw (S5-5) node[below=8pt, right=-7pt,black]{\nmm 5};
\draw (S41-5) node[below=8pt, right=-7pt,black] {\nmm 5};
\draw (S32-5) node[below=8pt, right=0pt,black] {\nmm 1};
\draw (S311p-5) node[below=8pt, right=-5pt,black] {\nmm 1};
\draw (S311m-5) node[below=8pt, right=-5pt,black] {\nmm 1};
\path (-10,-30)  node[anchor=west]  {$\bsym{\ell=3}$};
\path (0,-30)  node[anchor=west] (S6-6) {\yng(6)};
\path (12,-30)  node[anchor=west] (S51-6) {\yng(5,1)};
\path (23,-30)  node[anchor=west] (S42-6) {\yng(4,2)};
\path (33,-30)  node[anchor=west] (S411-6) {\yng(4,1,1)};
\path (42,-30)  node[anchor=west] (S33-6) {\yng(3,3)};
\path (49,-30)  node[anchor=west] (S321p-6) {$\yng(3,2,1)^{\bsym{+}}$};
\path (58,-30)  node[anchor=west] (S321m-6) {$\yng(3,2,1)^{\bsym{-}}$};
\draw (S6-6)      node[below=8pt, right=-7pt,black]{\nmm 5};
\draw (S51-6)    node[below=8pt, right=-10pt,black] {\nmm 10};
\draw (S42-6)    node[below=8pt, right=-4pt,black] {\nmm 6};
\draw (S411-6)   node[below=8pt, right=-7pt,black]  {\nmm 7};
\draw (S33-6) node[below=12pt, right=-9pt,black] {\nmm 1};
\draw (S321p-6) node[below=12pt, right=-5pt,black] {\nmm 2};
\draw (S321m-6) node[below=8pt, right=-5pt,black] {\nmm 2};
\path (-10,-35)  node[anchor=west]  {$\bsym{\ell=\frac{7}{2}}$};
\path (6,-35)  node[anchor=west] (S5-7) {\yng(5)};
\path (19,-35)  node[anchor=west] (S41-7) {\yng(4,1)};
\path (35,-35)  node[anchor=west] (S32-7) {\yng(3,2)};
\path (49,-35)  node[anchor=west] (S311p-7) {$\yng(3,1,1)^{\bsym{+}}$};
\path (58,-35)  node[anchor=west] (S311m-7) {$\yng(3,1,1)^{\bsym{-}}$};
\draw (S5-7)  node[below=8pt, right=2pt,black] {\nmm 15};
\draw (S41-7)  node[below=6pt, right=-7pt,black] {\nmm 23};
\draw (S32-7)  node[below=12pt, right=-7pt,black] {\nmm 11};
\draw (S311p-7)  node[below=8pt, right=-4pt,black] {\nmm 9};
\draw (S311m-7)  node[below=8pt, right=-4pt,black] {\nmm 9};
\path (-10,-40)  node[anchor=west]  {$\bsym{\ell=4}$};
\path (0,-40)  node[anchor=west] (S6-8) {\yng(6)};
\path (12,-40)  node[anchor=west] (S51-8) {\yng(5,1)};
\path (23,-40)  node[anchor=west] (S42-8) {\yng(4,2)};
\path (33,-40)  node[anchor=west] (S411-8) {\yng(4,1,1)};
\path (42,-40)  node[anchor=west] (S33-8) {\yng(3,3)};
\path (49,-40)  node[anchor=west] (S321p-8) {$\yng(3,2,1)^{\bsym{+}}$};
\path (58,-40)  node[anchor=west] (S321m-8) {$\yng(3,2,1)^{\bsym{-}}$};
\draw (S6-8)          node[below=12pt, right=-4pt,black]  {\nmm 15};
\draw (S51-8)        node[below=12pt, right=-4pt,black] {\nmm 38};
\draw (S42-8)        node[below=12pt, right=-4pt,black] {\nmm 34};
\draw (S411-8)      node[below=12pt, right=-4pt,black] {\nmm 41};
\draw (S33-8)        node[below=12pt, right=-4pt,black] {\nmm 11};
\draw (S321p-8)      node[below=12pt, right=-4pt,black] {\nmm 20};
\draw (S321m-8)      node[below=12pt, right=-4pt,black] {\nmm 20};
\draw (80, 0) node[anchor=east,black] {$~~{\mathbf{1}}$};
\draw (80,-5) node[anchor=east,black] {$~~{\mathbf{1}}$};
\draw (80,-10) node[anchor=east,black] {$~~{\mathbf{2}}$};
\draw (80,-15) node[anchor=east,black] {$~~{\mathbf{5}}$};
\draw (80,-20) node[anchor=east,black] {$~~{\mathbf{15}}$};
\draw (80,-25) node[anchor=east,black] {$~~{\mathbf{53}}$};
\draw (80,-30) node[anchor=east,black] {$~~{\mathbf{219}}$};
\draw (80,-35) node[anchor=east,black] {$~~{\mathbf{1037}}$};
\draw (80,-40) node[anchor=east,black] {$~~{\mathbf{5427}}$};
\path  (S6-0) edge[black,thick] (S5-1);
\path  (S6-2) edge[black,thick] (S5-1);
\path  (S51-2) edge[black,thick] (S5-1);
\path  (S6-2) edge[black,thick] (S5-3);
\path  (S51-2) edge[black,thick] (S5-3);
\path  (S51-2) edge[black,thick] (S41-3);
\path  (S6-4) edge[black,thick] (S5-3);
\path  (S51-4) edge[black,thick] (S5-3);
\path  (S51-4) edge[black,thick] (S41-3);
\path  (S42-4) edge[black,thick] (S41-3);
\path  (S411-4) edge[black,thick] (S41-3);
\path  (S6-4) edge[black,thick] (S5-5);
\path  (S51-4) edge[black,thick] (S5-5);
\path  (S51-4) edge[black,thick] (S41-5);
\path  (S42-4) edge[black,thick] (S41-5);
\path  (S42-4) edge[black,thick] (S32-5);
\path  (S411-4) edge[black,thick] (S41-5);
\path  (S411-4) edge[black,thick] (S311p-5);
\path  (S411-4) edge[black,thick] (S311m-5);
\path  (S6-6) edge[black,thick] (S5-5);
\path  (S51-6) edge[black,thick] (S5-5);
\path  (S51-6) edge[black,thick] (S41-5);
\path  (S42-6) edge[black,thick] (S41-5);
\path  (S42-6) edge[black,thick] (S32-5);
\path  (S411-6) edge[black,thick] (S41-5);
\path  (S411-6) edge[black,thick] (S311p-5);
\path  (S411-6) edge[black,thick] (S311m-5);
\path  (S33-6) edge[black,thick] (S32-5);
\path  (S321p-6) edge[black,thick] (S32-5);
\path  (S321m-6) edge[black,thick] (S32-5);
\path  (S321p-6) edge[black,thick] (S311p-5);
\path  (S321m-6) edge[black,thick] (S311m-5);
\path  (S6-6) edge[black,thick] (S5-7);
\path  (S51-6) edge[black,thick] (S5-7);
\path  (S51-6) edge[black,thick] (S41-7);
\path  (S42-6) edge[black,thick] (S41-7);
\path  (S42-6) edge[black,thick] (S32-7);
\path  (S33-6) edge[black,thick] (S32-7);
\path  (S321p-6) edge[black,thick] (S32-7);
\path  (S321m-6) edge[black,thick] (S32-7);
\path  (S321p-6) edge[black,thick] (S311p-7);
\path  (S321m-6) edge[black,thick] (S311m-7);
\path  (S411-6) edge[black,thick] (S41-7);
\path  (S411-6) edge[black,thick] (S311p-7);
\path  (S411-6) edge[black,thick] (S311m-7);
\path  (S6-8) edge[black,thick] (S5-7);
\path  (S51-8) edge[black,thick] (S5-7);
\path  (S51-8) edge[black,thick] (S41-7);
\path  (S42-8) edge[black,thick] (S41-7);
\path  (S411-8) edge[black,thick] (S41-7);
\path  (S42-8) edge[black,thick] (S32-7);
\path  (S33-8) edge[black,thick] (S32-7);
\path  (S321p-8) edge[black,thick] (S32-7);
\path  (S411-8) edge[black,thick] (S311p-7);
\path  (S321p-8) edge[black,thick] (S311p-7);
\path  (S411-8) edge[black,thick] (S311p-7);
\path  (S321m-8) edge[black,thick] (S32-7);
\path  (S411-8) edge[black,thick] (S311m-7);
\path  (S321m-8) edge[black,thick] (S311m-7);
\path  (S411-8) edge[black,thick] (S311m-7);
\end{tikzpicture}
$$

\newpage
  
  \subsection{Levels  \ $\ell = 0,\half,1,\ldots,\frac{7}{2},4$ of the quasi-Bratteli diagram $\mathcal{QB}(\A_6,\A_5)$} \label{sec:QBratteliA6A5}
\bigskip

To calculate the subscripts on the half-integer rows, use Pascal addition of the subscripts from the row above. To calculate the subscripts on integer level rows, first use Pascal addition from the row above,  and then subtract the subscript on the same partition from two rows above.
\medskip
  
\Yboxdim{6pt}
\Ylinethick{.6pt}

$$
\begin{tikzpicture}[line width=.5pt,xscale=0.16,yscale=0.35]
\path (-10,0)  node[anchor=west]  {$\bsym{\ell=0}$};
\path (0,0)  node[anchor=west] (S6-0) {\yng(6)};
\draw (S6-0) node[below=4pt, black] {\nmm 1};
\path (-10,-5)  node[anchor=west]  {$\bsym{\ell=\frac{1}{2}}$};
\path (6,-5)  node[anchor=west] (S5-1) {\yng(5)};
\draw (S5-1) node[below=4pt,black] {\nmm 1};
\path (-10,-10)  node[anchor=west]  {$\bsym{\ell=1}$};
\path (0,-10)  node[anchor=west] (S6-2) {\yng(6)};
\path (12,-10)  node[anchor=west] (S51-2) {\yng(5,1)};
\draw (S6-2) node[below=8pt,right=-7pt,black] {\nmm 0};
\draw (S51-2) node[below=8pt, right=-6pt,black] {\nmm 1};
\path (-10,-15)  node[anchor=west]  {$\bsym{\ell=\frac{3}{2}}$};
\path (6,-15)  node[anchor=west] (S5-3) {\yng(5)};
\path (19,-15)  node[anchor=west] (S41-3) {\yng(4,1)};
\draw (S5-3) node[below=8pt, right=-7pt,black] {\nmm 1};
\draw (S41-3) node[below=8pt, right=-6pt,black] {\nmm 1};
\path (-10,-20)  node[anchor=west]  {$\bsym{\ell=2}$};
\path (0,-20)  node[anchor=west] (S6-4) {\yng(6)};
\path (12,-20)  node[anchor=west] (S51-4) {\yng(5,1)};
\path (23,-20)  node[anchor=west] (S42-4) {\yng(4,2)};
\path (33,-20)  node[anchor=west] (S411-4) {\yng(4,1,1)};
\draw (S6-4)    node[below=8pt, right=-7pt,black]{\nmm 1};
\draw (S51-4)  node[below=8pt, right=-7pt,black] {\nmm 1};
\draw (S42-4)  node[below=8pt, right=-3pt,black] {\nmm 1};
\draw (S411-4) node[below=8pt, right=-6pt,black] {\nmm 1};
\path (-10,-25)  node[anchor=west]  {$\bsym{\ell=\frac{5}{2}}$};
\path (6,-25)  node[anchor=west] (S5-5) {\yng(5)};
\path (19,-25)  node[anchor=west] (S41-5) {\yng(4,1)};
\path (35,-25)  node[anchor=west] (S32-5) {\yng(3,2)};
\path (49,-25)  node[anchor=west] (S311p-5) {$\yng(3,1,1)^{\bsym{+}}$};
\path (58,-25)  node[anchor=west] (S311m-5) {$\yng(3,1,1)^{\bsym{-}}$};
\draw (S5-5) node[below=8pt, right=-7pt,black]{\nmm 2};
\draw (S41-5) node[below=8pt, right=-7pt,black] {\nmm 3};
\draw (S32-5) node[below=8pt, right=0pt,black] {\nmm 1};
\draw (S311p-5) node[below=8pt, right=-5pt,black] {\nmm 1};
\draw (S311m-5) node[below=8pt, right=-5pt,black] {\nmm 1};
\path (-10,-30)  node[anchor=west]  {$\bsym{\ell=3}$};
\path (0,-30)  node[anchor=west] (S6-6) {\yng(6)};
\path (12,-30)  node[anchor=west] (S51-6) {\yng(5,1)};
\path (23,-30)  node[anchor=west] (S42-6) {\yng(4,2)};
\path (33,-30)  node[anchor=west] (S411-6) {\yng(4,1,1)};
\path (42,-30)  node[anchor=west] (S33-6) {\yng(3,3)};
\path (49,-30)  node[anchor=west] (S321p-6) {$\yng(3,2,1)^{\bsym{+}}$};
\path (58,-30)  node[anchor=west] (S321m-6) {$\yng(3,2,1)^{\bsym{-}}$};
\draw (S6-6)      node[below=8pt, right=-7pt,black]{\nmm 1};
\draw (S51-6)    node[below=8pt, right=-10pt,black] {\nmm 4};
\draw (S42-6)    node[below=8pt, right=-4pt,black] {\nmm 3};
\draw (S411-6)   node[below=8pt, right=-7pt,black]  {\nmm 4};
\draw (S33-6) node[below=12pt, right=-9pt,black] {\nmm 1};
\draw (S321p-6) node[below=12pt, right=-5pt,black] {\nmm 2};
\draw (S321m-6) node[below=8pt, right=-5pt,black] {\nmm 2};
\path (-10,-35)  node[anchor=west]  {$\bsym{\ell=\frac{7}{2}}$};
\path (6,-35)  node[anchor=west] (S5-7) {\yng(5)};
\path (19,-35)  node[anchor=west] (S41-7) {\yng(4,1)};
\path (35,-35)  node[anchor=west] (S32-7) {\yng(3,2)};
\path (49,-35)  node[anchor=west] (S311p-7) {$\yng(3,1,1)^{\bsym{+}}$};
\path (58,-35)  node[anchor=west] (S311m-7) {$\yng(3,1,1)^{\bsym{-}}$};
\draw (S5-7)  node[below=8pt, right=2pt,black] {\nmm 5};
\draw (S41-7)  node[below=6pt, right=-7pt,black] {\nmm 11};
\draw (S32-7)  node[below=12pt, right=-7pt,black] {\nmm 8};
\draw (S311p-7)  node[below=8pt, right=-4pt,black] {\nmm 6};
\draw (S311m-7)  node[below=8pt, right=-4pt,black] {\nmm 6};
\path (-10,-40)  node[anchor=west]  {$\bsym{\ell=4}$};
\path (0,-40)  node[anchor=west] (S6-8) {\yng(6)};
\path (12,-40)  node[anchor=west] (S51-8) {\yng(5,1)};
\path (23,-40)  node[anchor=west] (S42-8) {\yng(4,2)};
\path (33,-40)  node[anchor=west] (S411-8) {\yng(4,1,1)};
\path (42,-40)  node[anchor=west] (S33-8) {\yng(3,3)};
\path (49,-40)  node[anchor=west] (S321p-8) {$\yng(3,2,1)^{\bsym{+}}$};
\path (58,-40)  node[anchor=west] (S321m-8) {$\yng(3,2,1)^{\bsym{-}}$};
\draw (S6-8)          node[below=12pt, right=-4pt,black]  {\nmm 4};
\draw (S51-8)        node[below=12pt, right=-4pt,black] {\nmm 12};
\draw (S42-8)        node[below=12pt, right=-4pt,black] {\nmm 16};
\draw (S411-8)      node[below=12pt, right=-4pt,black] {\nmm 19};
\draw (S33-8)        node[below=12pt, right=-4pt,black] {\nmm 7};
\draw (S321p-8)      node[below=12pt, right=-4pt,black] {\nmm 12};
\draw (S321m-8)      node[below=12pt, right=-4pt,black] {\nmm 12};
\draw (75, 0) node[anchor=east,black] {$~~{\mathbf{1}}$};
\draw (75,-5) node[anchor=east,black] {$~~{\mathbf{1}}$};
\draw (75,-10) node[anchor=east,black] {$~~{\mathbf{1}}$};
\draw (75,-15) node[anchor=east,black] {$~~{\mathbf{2}}$};
\draw (75,-20) node[anchor=east,black] {$~~{\mathbf{4}}$};
\draw (75,-25) node[anchor=east,black] {$~~{\mathbf{16}}$};
\draw (75,-30) node[anchor=east,black] {$~~{\mathbf{51}}$};
\draw (75,-35) node[anchor=east,black] {$~~{\mathbf{282}}$};
\draw (75,-40) node[anchor=east,black] {$~~{\mathbf{1114}}$};
\path  (S6-0) edge[black,thick] (S5-1);
\path  (S6-2) edge[black,thick] (S5-1);
\path  (S51-2) edge[black,thick] (S5-1);
\path  (S6-2) edge[black,thick] (S5-3);
\path  (S51-2) edge[black,thick] (S5-3);
\path  (S51-2) edge[black,thick] (S41-3);
\path  (S6-4) edge[black,thick] (S5-3);
\path  (S51-4) edge[black,thick] (S5-3);
\path  (S51-4) edge[black,thick] (S41-3);
\path  (S42-4) edge[black,thick] (S41-3);
\path  (S411-4) edge[black,thick] (S41-3);
\path  (S6-4) edge[black,thick] (S5-5);
\path  (S51-4) edge[black,thick] (S5-5);
\path  (S51-4) edge[black,thick] (S41-5);
\path  (S42-4) edge[black,thick] (S41-5);
\path  (S42-4) edge[black,thick] (S32-5);
\path  (S411-4) edge[black,thick] (S41-5);
\path  (S411-4) edge[black,thick] (S311p-5);
\path  (S411-4) edge[black,thick] (S311m-5);
\path  (S6-6) edge[black,thick] (S5-5);
\path  (S51-6) edge[black,thick] (S5-5);
\path  (S51-6) edge[black,thick] (S41-5);
\path  (S42-6) edge[black,thick] (S41-5);
\path  (S42-6) edge[black,thick] (S32-5);
\path  (S411-6) edge[black,thick] (S41-5);
\path  (S411-6) edge[black,thick] (S311p-5);
\path  (S411-6) edge[black,thick] (S311m-5);
\path  (S33-6) edge[black,thick] (S32-5);
\path  (S321p-6) edge[black,thick] (S32-5);
\path  (S321m-6) edge[black,thick] (S32-5);
\path  (S321p-6) edge[black,thick] (S311p-5);
\path  (S321m-6) edge[black,thick] (S311m-5);
\path  (S6-6) edge[black,thick] (S5-7);
\path  (S51-6) edge[black,thick] (S5-7);
\path  (S51-6) edge[black,thick] (S41-7);
\path  (S42-6) edge[black,thick] (S41-7);
\path  (S42-6) edge[black,thick] (S32-7);
\path  (S33-6) edge[black,thick] (S32-7);
\path  (S321p-6) edge[black,thick] (S32-7);
\path  (S321m-6) edge[black,thick] (S32-7);
\path  (S321p-6) edge[black,thick] (S311p-7);
\path  (S321m-6) edge[black,thick] (S311m-7);
\path  (S411-6) edge[black,thick] (S41-7);
\path  (S411-6) edge[black,thick] (S311p-7);
\path  (S411-6) edge[black,thick] (S311m-7);
\path  (S6-8) edge[black,thick] (S5-7);
\path  (S51-8) edge[black,thick] (S5-7);
\path  (S51-8) edge[black,thick] (S41-7);
\path  (S42-8) edge[black,thick] (S41-7);
\path  (S411-8) edge[black,thick] (S41-7);
\path  (S42-8) edge[black,thick] (S32-7);
\path  (S33-8) edge[black,thick] (S32-7);
\path  (S321p-8) edge[black,thick] (S32-7);
\path  (S411-8) edge[black,thick] (S311p-7);
\path  (S321p-8) edge[black,thick] (S311p-7);
\path  (S411-8) edge[black,thick] (S311p-7);
\path  (S321m-8) edge[black,thick] (S32-7);
\path  (S411-8) edge[black,thick] (S311m-7);
\path  (S321m-8) edge[black,thick] (S311m-7);
\path  (S411-8) edge[black,thick] (S311m-7);
\end{tikzpicture}
$$ 

  
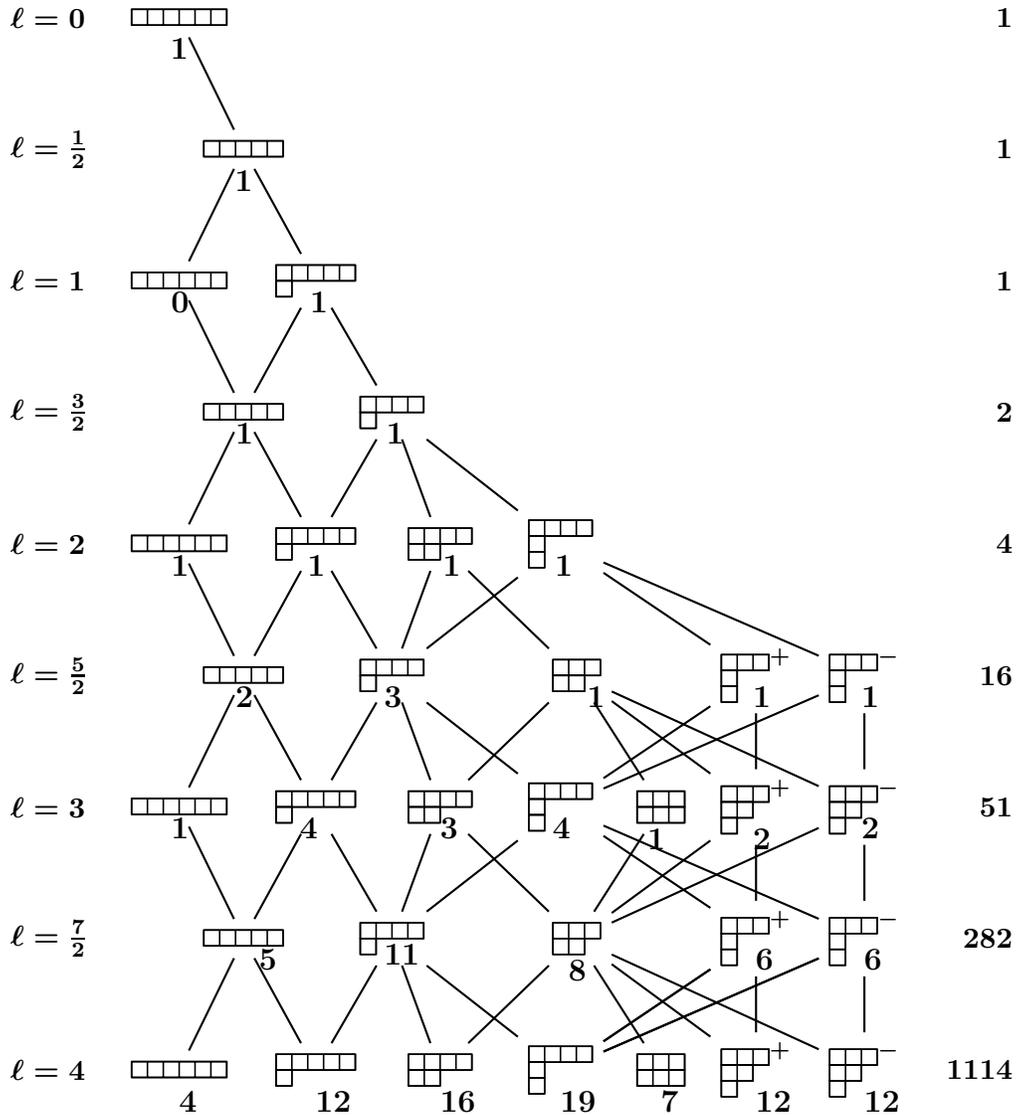
\end{document}